\documentclass[journal,twoside,web]{ieeecolor}
\usepackage{generic}
\usepackage{cite}
\usepackage{amsmath,amssymb,amsfonts}
\usepackage{algorithmic}
\usepackage{graphicx}
\usepackage{algorithm,algorithmic}
\usepackage{hyperref}
\hypersetup{hidelinks=true}
\usepackage{textcomp}
\usepackage{stmaryrd}
\usepackage{mathrsfs}
\usepackage{tikz}
\usetikzlibrary{shapes.geometric, arrows}
\graphicspath{{images/}}
\usepackage[colorinlistoftodos]{todonotes}
\usepackage{caption}
\usepackage{subcaption}
\usepackage{bm}
\usepackage[T1]{fontenc}
\usepackage{tcolorbox}
\usepackage{textcomp}
\usepackage{bigints}
\usepackage{cleveref}
\usepackage{siunitx}
\usepackage{textgreek}

\newcommand{\norm}[1]{\left\lVert#1\right\rVert}

\newcommand\xqed[1]{%
	\leavevmode\unskip\penalty9999 \hbox{}\nobreak\hfill
	\quad\hbox{#1}}
\newcommand\demo{\xqed{$\triangle$}}

\newcommand{\dt}{\mathrm{d}t}
\newcommand{\dhatt}{\mathrm{d}\hat{t}}

\newcommand\restr[2]{{
		\left.\kern-\nulldelimiterspace 
		#1 
		\littletaller 
		\right|_{#2} 
}}

\newcommand{\littletaller}{\mathchoice{\vphantom{\big|}}{}{}{}}

\definecolor{bananayellow}{rgb}{1.0, 0.88, 0.21}
\definecolor{bittersweet}{rgb}{1.0, 0.44, 0.37}

\newtheorem{definition}{Definition}
\newtheorem{theorem}{Theorem}
\newtheorem{proposition}{Proposition}
\newtheorem{lemma}{Lemma}

\newtheorem{example}{Example}
\newtheorem{remark}{Remark}

\def\BibTeX{{\rm B\kern-.05em{\sc i\kern-.025em b}\kern-.08em
    T\kern-.1667em\lower.7ex\hbox{E}\kern-.125emX}}
\begin{document}
\title{Comparison of Linear Systems Across Time Domains: Continuous-time vs. Discrete-time}
\author{Armin Pirastehzad \IEEEmembership{Member, IEEE} and Bart Besselink \IEEEmembership{Senior Member, IEEE}
\thanks{All authors are with the Bernoulli Institute for Mathematics, ComputerScience and Artificial Intelligence, University of Groningen, 9700 AK Groningen, The Netherlands (email: a.pirastehzad@rug.nl; b.besselink@rug.nl).}
}

\maketitle

\begin{abstract}
We develop a formal framework for the behavioral comparison of linear systems across different time domains. We accomplish this by introducing the notion of system interpolation, which determines whether the input-state trajectories of a continuous-time system can be realized as piecewise polynomial interpolations of the input-state trajectories of a discrete-time system. In this context, a piecewise polynomial interpolation of a discrete-time signal is characterized as a continuous-time function that coincides with the discrete-time signal at given sampling instants and can be realized as a polynomial of a prescribed degree over intervals between these instants. By representing piecewise polynomial functions as linear combinations of shifted Legendre polynomials, we characterize system interpolation as a subspace inclusion that is completely in terms of system parameters. This therefore allows for a computationally efficient comparison of the input-state behavior of a continuous-time system with that of a discrete-time one. We then exploit this characterization to discretize a given continuous-time system into a discrete-time one. Lastly, given a control specification, we exploit system interpolation to synthesize controllers that ensure satisfaction at each given sampling instant, while they measure the extent of (possible) violation over intervals between these instants. 
\end{abstract}

\section{Introduction}\label{Sec_Introduction}
\IEEEPARstart{M}{odern} engineering applications often necessitate mobilization of  \emph{autonomous} systems, which are designed to operate in (uncertain) environments without continuous human intervention. Particularly, these systems are deployed in safety-critical technologies (\textit{e.g.,} self-driving vehicles \cite{78231091}), mission-critical utilities (\textit{e.g.,} power grids \cite{DOSSANTOSSERRA20251116251}), and medical-critical services (\textit{e.g.,} healthcare robots \cite{Tavakoli20241}). These applications often require autonomous systems to perform in accordance with intricate specifications, such as those expressed in terms of temporal logic formulae \cite{fisher2011introduction}. 

For continuous-time systems, such intricate specifications are usually enforced by controllers that are synthesized according to a \emph{hierarchical} procedure \cite{YIN2024100940}. In particular, many control schemes (see, \textit{e.g.,} \cite{4459804,MEYER2018437,8657716,8462741,8010291,POLA2019178,6082386,mazo2024contracttheorylayeredcontrol}) split synthesis into 1) \emph{discretization}, where a simplified discrete-time model of the continuous-time system is obtained, 2) \emph{planning}, through which a discrete control sequence is designed to enforce the discrete-time model to fulfill the specification, and 3) \emph{execution}, which involves the synthesis of a continuous-time control that enforces the state trajectories of the original system to remain sufficiently close to the \emph{planned} discrete-time trajectories.

Since these hierarchical schemes enforce the continuous-time state trajectories to remain sufficiently close to \emph{interpolations} of the planned discrete-time state trajectories, they usually guarantee the specification only at certain instances of time rather than at all times. In fact, such schemes are inconclusive on how controlled continuous-time trajectories behave with respect to the specification in between discrete time instances, \textit{i.e.,} it is not clear to what extent the controlled continuous-time trajectories satisfy/violate the specification between these instances. Measuring the extent of satisfaction/violation of a specification over intervals between instances necessitates a metric that rigorously compares the behavioral similarity of the controlled continuous-time and the planned discrete-time trajectories. This paper proposes such metric by developing a \emph{formal} framework for comparison of the input-state behavior of linear systems from continuous-time and discrete-time domains. This framework is then utilized for control synthesis to 1) ensure that the controlled continuous-time trajectories adhere to the specification at time instances that correspond to the planned discrete-time trajectories, and 2) measure the extent to which they may violate the specification in between these instances. 
\subsection{Related Works}
Behavioral similarity of dynamical systems has been considered from a variety of perspectives. Inspired by concurrency theory, the notion of \emph{(bi)simulation} \cite{lafferriere1997hybrid,pappas2000hierarchically,pappas2003bisimilar} utilizes the formalism of labeled transition systems to determine the ability of a system to reproduce the input-output behavior of another. The notion of (bi)simulation is further specialized to continuous-time systems in \cite{van2004equivalence}, where it is formulated as a modified disturbance decoupling problem and characterized algebraically. Generalization of (bi)simulation to control systems gave rise to the notion of \emph{alternating (bi)simulation} \cite{tabuada2008controller} which determines equivalence of the controlled input-output behavior of two systems. The notion of (bi)simulation and its alternating version are further utilized for specification verification and control synthesis in \cite{tabuada2009verification}.  

As (bi)simulation fails to compare systems with similar (but not identical) input-output behaviors, the notion of \emph{approximate (bi)simulation} was proposed in \cite{girard2005approximate,girard2007approximation,girard2007approximate,girard2008approximate} to quantify behavioral \emph{closeness}, and was later specialized to verification and synthesis in \cite{fiore2018approximate,pola2017approximate,pola2009symbolic,julius2009approximate}. However, because approximate (bi)simulation relies on the $\mathcal{L}_\infty$ signal norm to measure behavioral closeness, it is \emph{not} compatible with many analytic/synthetic tools in control theory, which often employ the $\mathcal{L}_2$ signal norm. To address this limitation, the notion of $(\gamma,\delta)$-similarity was proposed in \cite{pirastehzad2021,pirastehzad2023}, which makes use of the $\mathcal{L}_2$ signal norm to measure this closeness.

While all the notions mentioned above provide a formal methodology for system comparison, they share the limitation that they only compare systems within a \emph{single} time domain, \textit{i.e.,} they either compare continuous-time systems or discrete-time ones. Even in the case of hybrid systems, the notion of (bi)simulation (and its approximate variants) compares the continuous-time dynamics associated with each discrete mode in isolation, \textit{i.e.,} it compares the continuous-time dynamics for fixed discrete modes (see, \textit{e.g.,} \cite{girard2008approximate,tabuada2009verification}).

Many hierarchical schemes for temporal logic control (\textit{e.g.,} \cite{4459804,7798307,Liu2014,LIU20161,10189357,9696363}), however, involve an \emph{implicit} comparison of the external behavior of a continuous-time system with that of its discrete-time model, as they aim to provide guarantees on the continuous-time system according to strategies planned for its discrete-time model. However, such comparison highly depends on the specification, as the behavioral similarity of the continuous-time system and its discrete-time model is measured with respect to a given specification. In fact, none of these schemes provide a `context-independent' formal methodology for comparison of the behavioral similarity of a continuous-time system and its discrete-time model.  

A context-independent comparison of continuous-time systems and their discrete-time models is crucial, as many control problems can be efficiently addressed in discrete time, \textit{e.g.,} control for temporal logic specification \cite{7039363,7798307} and optimization-based approaches such as model predictive control (MPC) \cite{kouvaritakis2016model}. Despite this importance, only a few methodologies for such comparison have been proposed. Exploiting functional transformations to convert signals from one time domain to another, the notion of \emph{(alternating) $\mathcal{F}^{-1}$-simulation} \cite{mazo2024contracttheorylayeredcontrol} determines whether a system is able to produce a transformation of the input-output behavior of another system, \textit{i.e.,} it determines whether the behavior of one system can be transformed into that of another. The notion of (alternating) $\mathcal{F}^{-1}$-simulation can be formulated within the framework of (alternating) simulation. For finite-state systems, such formulation allows for the adoption of fixed point theorems (see, \textit{e.g.,} \cite[Theorem 5.6]{tabuada2009verification}) to obtain a computationally efficient characterization of (alternating) $\mathcal{F}^{-1}$-simulation. For infinite-state systems, however, characterization of (alternating) $\mathcal{F}^{-1}$-simulation is computationally challenging. In fact, comparison of such systems according to (alternating) $\mathcal{F}^{-1}$-simulation requires state abstraction, which becomes computationally intractable as the system dimension grows. 
 
The goal of this paper is therefore to develop a \emph{formal} framework for \emph{efficiently} comparing the input-state behavior of (infinite-state) linear systems across different time domains, namely the continuous and discrete one.
\subsection{Contributions}
First, we introduce the notion of \emph{system interpolation}, which determines whether the input-state trajectories of a continuous-time system can be realized as piecewise polynomial interpolations of the input-state trajectories of a discrete-time system. Specifically, for a given sampling time and a prescribed polynomial degree, we identify a continuous-time system as an \emph{interpolator} of a discrete-time one if, for any discrete input sequence to the discrete-time system, there exists a piecewise polynomial interpolation of this sequence subject to which the continuous-time system admits a state trajectory which is also a piecewise polynomial interpolation of the corresponding state trajectory of the discrete-time system. Here, we identify a piecewise polynomial interpolation of a discrete-time signal as a continuous-time function such that 1) at each sampling instant, the continuous-time function coincides with the discrete-time signal and 2) over each interval between two consecutive instants, the continuous-time function is represented by a polynomial of the prescribed degree. 

Second, we obtain an algebraic characterization of system interpolation that solely depends on parameters of the continuous-time and discrete-time systems. To this end, we demonstrate that system interpolation can be investigated by considering discrete-time input-state trajectories over a \emph{single} (discrete) time step, \textit{i.e.,} in order to inspect system interpolation, it suffices to check whether the continuous-time system can generate piecewise polynomial interpolations of discrete-time input-state trajectories over a single interval. We accordingly characterize these piecewise polynomial interpolations as linear combinations of the so-called \emph{shifted Legendre polynomials} (see, \textit{e.g.,} \cite{Canuto2006,wang2012legendre}). This then enables the characterization of system interpolation as a \emph{subspace inclusion} that completely depends on parameters of the continuous-time and discrete-time systems. Such characterization is computationally efficient since subspace inclusion can be formulated as a simple rank condition. 

Third, for a given input-state trajectory of the discrete-time system, we characterize the set of \emph{all} interpolating inputs that enforce the continuous-time system to admit input-state trajectories that are piecewise polynomial interpolations of this discrete-time trajectory. Particularly, we use the characterization of system interpolation to construct each input in this set.

Fourth, we use system interpolation to conduct discretization. Specifically, for a given sampling time and a prescribed polynomial degree, we \emph{discretize} a given continuous-time system into a discrete-time model in such a way that the former is an interpolator of the latter. We accomplish this by making use of the algebraic characterization of system interpolation to characterize the existence of such discrete-time model in terms of a linear matrix equation, whose solution gives the parameters of this model. 

Fifth, we utilize system interpolation for control synthesis. Specifically, for a given specification, we synthesize controllers on the basis of a discrete-time model of the continuous-time system (obtained via discretization using system interpolation) to 1) ensure that the controlled continuous-time trajectories adhere to the specification at \emph{each} sampling instant and 2) give a measure of the extent to which these trajectories may \emph{violate} the specification over intervals between consecutive sampling instances. In particular, by measuring (possible) specification violation in terms of the sampling time and the prescribed polynomial degree (of interpolation), we show how the extent of specification violation varies with respect to the sampling time and the polynomial degree. 

The rest of this paper is organized as follows. The problem statement is given in Section~\ref{Sec_ProblemStatement}. In    Section~\ref{Sec_LegendrePolynomials}, we briefly review the basic properties of the shifted Legendre polynomials, based upon which the main results of the paper are built. We then introduce and characterize the notion of system interpolation in Section~\ref{Sec_SystemInterpolation}. We accordingly use this characterization to conduct discretization in Section~\ref{Sec_Discretization}. Subsequently, we utilize system interpolation for control synthesis in Section~\ref{Sec_ControlSynthesis}. We further demonstrate our results in a numerical example in Section~\ref{Sec_Simulation} and finally conclude the paper by Section~\ref{Sec_Conclusion}.
\subsection{Notation} 
The mathematical notation adopted in this paper is largely standard. Nevertheless, for clarity, we briefly review the general conventions. 
\subsubsection*{Sets}
We respectively denote the set of real numbers and integers by $\mathbb{R}$ and $\mathbb{Z}$, while we use $\mathbb{R}_{>0}$ and $\mathbb{Z}_{>0}$ to respectively indicate the set of positive real numbers and integers. For real numbers $t_1\leq t_2$, we denote the closed real interval $\{t\in\mathbb{R} \vert t_1\leq t \leq t_2\}$ by $[t_1,t_2]$. For integers $k_1\leq k_2$, on the other hand, we use $\llbracket k_1,k_2 \rrbracket$ to denote the closed integer interval $\{k\in\mathbb{Z} \vert k_1\leq k \leq k_2\}$. We utilize $\mathbb{R}^n$ to denote the set of real vectors with $n$ components, and write $\mathbb{R}^{n\times m}$ for the set of real $n\times m$ matrices. Then, we denote the image of a matrix $M\in\mathbb{R}^{n\times m}$ by $\operatorname{Im} M$ and define it as $\operatorname{Im} M = \{y\in\mathbb{R}^n \vert \exists x\in\mathbb{R}^m: Mx = y\}$. 
\subsubsection*{Operators} We define the operator $\operatorname{col}(\cdot,\cdot)$ such that $\operatorname{col}(x_1,x_2) = (x_1^\top,x_2^\top)^\top$, for any vectors $x_1\in\mathbb{R}^{n_1}$ and $x_2\in\mathbb{R}^{n_2}$. The Euclidean norm of a vector $x\in\mathbb{R}^n$ is defined as $\lvert x \rvert = (x^\top x)^{\frac{1}{2}}$. Then, for a positive (semi-)definite matrix $M\in\mathbb{R}^{n\times m}$, we define the weighted Euclidean (semi-)norm as $\lvert x \rvert_M = (x^\top M x)^{\frac{1}{2}}$. We define the vectorization of a matrix $M =[M_1 \, M_2\,\cdots\,M_m]$, with columns $M_1,M_2,\cdots,M_m\in\mathbb{R}^n$, as $\operatorname{vec}(M) = \operatorname{col}(M_1,M_2,\cdots,M_m)$. We also utilize the symbol $\otimes$ to denote the Kronecker product. Indicating an identity matrix and a zero matrix of appropriate dimensions by $I$ and $0$, respectively, we define the Kronecker sum of square matrices $N\in\mathbb{R}^{n\times n}$ and $M\in\mathbb{R}^{m\times m}$ as $N\oplus M = N\otimes I + I\otimes M$. 
\subsubsection*{Function Spaces and Operations}
For real numbers $\tau_1\leq \tau_2$, we define the space of $n$-dimensional vector-valued functions that are square-integrable over the interval $[\tau_1,\tau_2]$ as $\mathcal{L}_2^n [\tau_1,\tau_2] = \{ u: [\tau_1,\tau_2]\rightarrow\mathbb{R}^n  \vert  \int_{\tau_1}^{\tau_2} |u(t)|^2\dt < \infty\}$, which is endowed with the norm $\norm{u}_{\mathcal{L}_2^n[\tau_1,\tau_2]} = (\int_{\tau_1}^{\tau_2} |u(t)|^2\dt)^{\frac{1}{2}}$. We also define the space of $n$-dimensional vector-valued functions that are bounded over $[\tau_1,\tau_2]$ as $\mathcal{L}_\infty^n[\tau_1,\tau_2] = \{u: [\tau_1,\tau_2]\rightarrow\mathbb{R}^n \vert \max_{t\in[\tau_1,\tau_2]} |u(t)|<\infty \}$, equipped with the norm $\norm{u}_{\mathcal{L}_\infty^n[\tau_1,\tau_2]} = \max_{t\in[\tau_1,\tau_2]} |u(t)|$. Given a set $I$, we define the restriction of a function $u : I \rightarrow \mathbb{R}^n$ to the subset $I'\subset I$ as the function $\restr{u}{I'} : I' \rightarrow \mathbb{R}^n$ such that $\forall t \in I' : \restr{u}{I'} (t) = u(t)$. Then, for a set $\mathcal{F}$ of functions $f:I \rightarrow \mathbb{R}^n$, we define the restriction of $\mathcal{F}$ to $I'$ as 
$\restr{\mathcal{F}}{I'} = \left\{\restr{f}{I'} \big\vert f\in\mathcal{F}\right\}$. 

\section{Problem Statement}\label{Sec_ProblemStatement}
For continuous-time systems, control synthesis for enforcing intricate specifications is usually split into (at least) the following consecutive tasks (see, \textit{e.g.,} \cite{MEYER2018437,8657716,8462741,8010291,POLA2019178,6082386,mazo2024contracttheorylayeredcontrol}).
\begin{enumerate}
	\item \emph{Discretization}: First, a simplified \emph{discrete-time} model of the continuous-time system is constructed, as demonstrated on the left side of Figure~\ref{Fig_Hierarchy}. Such a model is obtained either according to the occurrence of particular \emph{events}, such as evolution of state trajectories from one region of interest to another (see, \textit{e.g.,} \cite{4459804}), or through discretization of system dynamics with respect to a \emph{sampling} time (see, \textit{e.g.,} \cite{POLA2019178}). 
	\item \emph{Planning}: Then, a discrete control sequence is designed to enforce the discrete-time model to fulfill the specification (see, \textit{e.g.,} \cite{MEYER2018437,8657716}). In fact, such control enforces the discrete-time model to admit a state trajectory that adheres to a discretized version of the original (continuous) specification, see the lower part of Figure~\ref{Fig_Hierarchy}.
	\item \emph{Execution}: Lastly, a continuous-time control is synthesized that enforces the original system to admit a state trajectory that is sufficiently \emph{close} to an interpolation of the planned discrete-time trajectory, see the middle part of Figure~\ref{Fig_Hierarchy}. Such control then enforces the system to (approximately) satisfy the specification, \textit{e.g.,} it may enforce the system to admit a trajectory whose samples adhere to the specification (see, \textit{e.g.,} \cite{POLA2019178,mazo2024contracttheorylayeredcontrol
	}).
\end{enumerate}
The prevalence of such hierarchical schemes stems from the fact that control synthesis for enforcing complex specifications is \emph{easier} in the discrete-time than in the continuous-time. While these schemes guarantee that a discrete-time model of a continuous-time system satisfies a specification, they do not \emph{formally} guarantee that the system itself \emph{fully} satisfies the specification. This is a consequence of the fact that these schemes synthesize a controller that renders the continuous-time system to admit state trajectories that remain (sufficiently) close to interpolations of the state trajectories of its discrete-time model. In other words, despite enforcing the specification at certain instances (\textit{e.g.,} sampled times), such a controller does not necessarily ensure satisfaction at \emph{all} times. In fact, it is not clear how the controlled continuous-time system behaves with respect to the specification in \emph{between} these instances, \textit{i.e.,} there is no measure of the extent to which the specification is satisfied/violated over intervals between the instances. 

Motivated by this, we develop a formal framework for comparison of the input-state behavior of a continuous-time system with that of a discrete-time one. Specifically, we consider the continuous-time linear system 
\begin{equation}\label{ContinuousSystem}
	\bm{\Sigma}_c: \begin{aligned}
		\dot{x}_c(t) &= A_cx_c(t) + B_cu_c(t),
	\end{aligned} 
\end{equation}
with state $x_c\in\mathbb{R}^n$ and input $u_c \in \mathbb{R}^m$. We denote by $x_c(t;x_{0},u_c)$ the state solution, at time $t$, of \eqref{ContinuousSystem} for initial condition $x_c(0) = x_{0}$ and input $u_c$.

In correspondence to \eqref{ContinuousSystem}, we consider the discrete-time linear system
\begin{equation}\label{DiscreteSystem}
	\bm{\Sigma}_d: \begin{aligned}
		x_d(i+1) &= A_dx_d(i) + B_du_d(i),
	\end{aligned} 
\end{equation}
with state $x_d\in\mathbb{R}^n$ and $u_d\in\mathbb{R}^m$. We use notation similar to that of \eqref{ContinuousSystem} to denote the state sequence of \eqref{DiscreteSystem}. 

Our goal is therefore to conceive a method that compares the input-state behavior of \eqref{ContinuousSystem} with that of \eqref{DiscreteSystem}. We will accomplish this by performing (piecewise) polynomial interpolation on the basis of Legendre polynomials, which are studied in the next section. 
\begin{figure}
	\centering
	\begin{tikzpicture}[scale=0.3]
		\node  (ContSys) at (0,0) [rectangle, 
		minimum width=0.3cm, 
		minimum height=0.5cm,
		text centered, 
		draw=black, 
		line width = 0.5pt,
		fill=gray! 10!, align=center] {\hspace*{-3pt}\fontsize{7}{30} \selectfont Continuous-time \\[-3pt] \hspace*{-3pt}\fontsize{7}{30} \selectfont System};
		
		\node  (ContCont) at (10,0) [rectangle, 
		minimum width=0.3cm, 
		minimum height=0.5cm,
		text centered, 
		draw=black, 
		line width = 0.5pt,
		fill=gray! 10!, align=center] {\hspace*{-3pt}\fontsize{7}{30} \selectfont Continuous-time \\[-3pt] \hspace*{-3pt}\fontsize{7}{30} \selectfont Control};
		\draw[line width = 0.5pt,->,>=stealth] (3.2,0.5)--(6.8,0.5);
		\draw[line width = 0.5pt,->,>=stealth] (6.8,-0.5)--(3.2,-0.5);
		
		\node  (int) at (14.5,0) {\fontsize{8}{30} \selectfont $\models$};
		
		\node  (Spec) at (19,0) [rectangle, 
		minimum width=0.3cm, 
		minimum height=0.5cm,
		text centered, 
		draw=black, 
		line width = 0.5pt,
		fill=gray! 10!, align=center] {\hspace*{-3pt}\fontsize{7}{30} \selectfont Continuous \\[-3pt]\hspace*{-3pt}\fontsize{7}{30} \selectfont Specification};
		
		\draw[line width = 0.7pt,-,dotted] (-3.6,1.9) --(22,1.9) -- (22,-1.9) -- (-3.6,-1.9) -- (-3.6,1.9);
		
		\node  (DiscSys) at (0,-8) [rectangle, 
		minimum width=0.3cm, 
		minimum height=0.5cm,
		text centered, 
		draw=black, 
		line width = 0.5pt,
		fill=gray! 10!, align=center] {\hspace*{-3pt}\fontsize{7}{30} \selectfont Discrete-time \\[-3pt] \hspace*{-3pt}\fontsize{7}{30} \selectfont System};
		
		\node  (DiscCont) at (10,-8) [rectangle, 
		minimum width=0.3cm, 
		minimum height=0.5cm,
		text centered, 
		draw=black, 
		line width = 0.5pt,
		fill=gray! 10!, align=center] {\hspace*{-3pt}\fontsize{7}{30} \selectfont Discrete-time \\[-3pt] \hspace*{-3pt}\fontsize{7}{30} \selectfont Control };
		\draw[line width = 0.5pt,->,>=stealth] (2.65,-7.5)--(7.35,-7.5);
		\draw[line width = 0.5pt,->,>=stealth] (7.35,-8.5)--(2.65,-8.5);
		
		\node  (int) at (14.5,-8) {\fontsize{8}{30} \selectfont $\models$};
		
		\node  (Spec) at (19,-8) [rectangle, 
		minimum width=0.3cm, 
		minimum height=0.5cm,
		text centered, 
		draw=black, 
		line width = 0.5pt,
		fill=gray! 10!, align=center] {\hspace*{-3pt}\fontsize{7}{30} \selectfont Discrete \\[-3pt]\hspace*{-3pt}\fontsize{7}{30} \selectfont Specification};
		\draw[line width = 0.5pt,->,>=stealth, dashed] (ContSys)--(DiscSys);
		\node  [rotate = 90] (discretization) at (-0.5,-4.2) {\fontsize{5}{30} \selectfont Discretization};
		
		\draw[line width = 0.5pt,->,>=stealth, dashed] (DiscCont)--(ContCont);
		\node  [rotate = 90] (exec) at (9.5,-4.2) {\fontsize{5}{30} \selectfont Execution};
		
		\draw[line width = 0.7pt,-,dotted] (-3,-6.1) --(22,-6.1) -- (22,-9.9) -- (-3,-9.9) -- (-3,-6.1);
		\node (planing) at (9.5,-10.5) {\fontsize{6}{30} \selectfont Planing};
		
	\end{tikzpicture}
	\caption{Control synthesis for enforcing intricate specifications is usually split into 1) discretization of the continuous-time system into a simplified discrete-time model (see left side of the figure), 2) planning a trajectory (by designing a suitable discrete control sequence) for the discrete-time model that adheres to the specification (see the lower part of the figure), and 3) execution of a continuous-time control that utilizes an interpolation of the planned discrete-time trajectory to enforce the original system to (approximately) satisfy the specification (see the middle part of the figure).}\label{Fig_Hierarchy}
\end{figure}
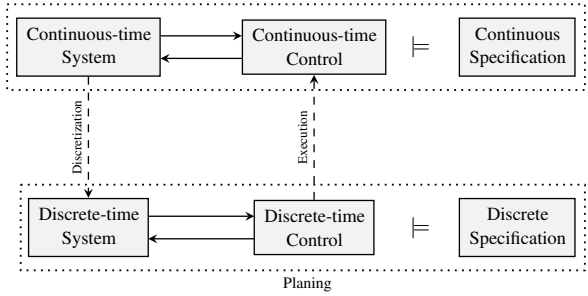

\section{Legendre Polynomials}\label{Sec_LegendrePolynomials}
We briefly review basic properties of Legendre polynomials, which are extensively utilized for polynomial interpolation as they form an \emph{orthogonal} basis for the space of polynomials (for further details on orthogonality, see, \textit{e.g.,} \cite{szeg1939orthogonal}). 

Given any integer $k\geq 0$, let $\mathfrak{L}_k$ be the \emph{standard} Legendre polynomial of degree $k$, defined as
\begin{equation}\label{StandardLegendre}
	\mathfrak{L}_k(t) = \frac{1}{2^k} \sum_{i = 0}^{\lfloor \frac{k}{2}\rfloor} (-1)^i \begin{pmatrix}
		k\\i
	\end{pmatrix} \begin{pmatrix}
		2k-2i\\k
	\end{pmatrix} t^{k-2i},
\end{equation}
where $\lfloor \frac{k}{2}\rfloor$ is the greatest integer less than or equal to $\frac{k}{2}$ (for an elaborate study of standard Legendre polynomials, see, \textit{e.g.,} \cite{guo1998spectral, Canuto2006, shen2011spectral}). Accordingly, for $\tau\in\mathbb{R}_{>0}$, we define the \emph{shifted} Legendre polynomial, of degree $k$, over the interval $[0,\tau]$ as
\begin{equation*}
	\mathfrak{L}_k^\tau (t) = \mathfrak{L}_k\left(\frac{2t}{\tau} -1\right). 
\end{equation*}
It then directly follows from the mutual orthogonality of standard Legendre polynomials over the interval $[-1,1]$ (see, \textit{e.g.,} \cite[Section 2.2.2]{Canuto2006}) that for all integers $k,l\geq 0$, 
\begin{equation}\label{MutualOrthogonality}
	\int_{0}^{\tau} \mathfrak{L}_k^\tau (t)\mathfrak{L}_l^\tau(t)\dt = \left\{\begin{aligned}
		&0&k\neq l,\\
		&\frac{\tau}{2k+1}&k=l,
	\end{aligned}\right.
\end{equation}
which indicates that shifted Legendre polynomials are \emph{linearly independent} over the interval $[0,\tau ]$.  

Now, for any integer $N\geq 0$, we let $\mathbb{P}_N^{n}$ denote the space of all $n$-dimensional vector-valued polynomials up to degree $N$. We thus conclude from the linear independence of shifted Legendre polynomials that 
\begin{equation}\label{PolynomialSpace}
	\restr{\mathbb{P}_N^n}{[0,\tau]} = \left\{\left.\sum_{k=0}^{N} \alpha_k\restr{\mathfrak{L}_k^\tau}{[0,\tau]} \right\vert \alpha_0,\alpha_1,\cdots,\alpha_N\in \mathbb{R}^n\right\},
\end{equation}
\textit{i.e.,} the restriction of any polynomial function (up to degree $N$) to the interval $[0,\tau]$ can be written as a linear combination of the restriction of shifted Legendre polynomials (up to degree $N$) to the interval $[0,\tau]$. We also note that $\restr{\mathbb{P}_N^{n}}{[0,\tau]}\subset\mathcal{L}_2^n[0,\tau]\cap \mathcal{L}_\infty^n[0,\tau]$. 

Subsequently, we construct a quadrature rule that enables exact integration of polynomial functions over the interval $[0,\tau]$. This requires a suitable choice of quadrature nodes and weights. For an integer $N\geq 1$, we choose the quadrature nodes $t_{N,0}<t_{N,1}<\ldots<t_{N,N}$ such that 
\begin{equation}\label{QuadratureNodes}
	\forall i\in\{0,1,\cdots,N\}: \quad \mathfrak{L}_N^\tau(t_{N,i}) + \mathfrak{L}_{N+1}^\tau(t_{N,i}) = 0,
\end{equation}
\textit{i.e.,} $t_{N,0}$, $t_{N,1}$, $\ldots$, $t_{N,N}$ are chosen as the zeros of $\mathfrak{L}_N^\tau + \mathfrak{L}_{N+1}^\tau$. We emphasize that such quadrature nodes always exists as $\mathfrak{L}_N^\tau + \mathfrak{L}_{N+1}^\tau$ is a \emph{separable} polynomial of degree $N+1$, \textit{i.e.,} it has $N+1$ distinct roots (for more details, see \cite[Section 2.2.3]{Canuto2006}). We further note that $t_{N,0} = 0$ and $t_{N,N}<\tau$ (see, \textit{e.g.,} \cite[Section 2]{wang2012legendre}).
  
Corresponding to these nodes, we define the quadrature weights  
\begin{subequations}\label{QuadratureWeight}
	\begin{align}
		w_{N,0} &= \frac{\tau}{(N+1)^2},
	\intertext{and}
	w_{N,i} &= \frac{1}{(N+1)^2}\frac{\tau-t_{N,i}}{|\mathfrak{L}_N^\tau(t_{N,i})|^2},
	\end{align}
\end{subequations}
for $i=1,2,\ldots,N$. It now follows from Gauss-Radau quadrature (see, \textit{e.g.,} \cite[Section 2.2.3]{Canuto2006}) that 
\begin{equation}\label{IntegrationComputation}
	\forall P\in\mathbb{P}_{2N}^{n}: \quad \int_{0}^{\tau}P(t)\dt = \sum_{i=0}^{N} P(t_{N,i})w_{N,i}.
\end{equation}

We note that \eqref{IntegrationComputation} enables computation of the $\mathcal{L}_2^n[0,\tau]$ norm of polynomial functions in terms of their values at the quadrature nodes. More specifically, for any function $P\in\restr{\mathbb{P}_N^{n}}{[0,\tau]}$, we employ \eqref{IntegrationComputation} to obtain 
\begin{equation}\label{NormComputation}
	\norm{P}_{\mathcal{L}_2^n[0,\tau]} = \left(\sum_{i=0}^{N}|P(t_{N,i})|^2w_{N,i}\right)^{\frac{1}{2}}.
\end{equation}

We further note that \eqref{IntegrationComputation} facilitates the representation of a function $P\in\restr{\mathbb{P}_N^{n}}{[0,\tau]}$ in terms of its values at quadrature nodes. Given $P\in\restr{\mathbb{P}_N^{n}}{[0,\tau]}$, it follows from \eqref{PolynomialSpace} that there exist vector-valued coefficients $\alpha_0,\alpha_1,\cdots,\alpha_N\in\mathbb{R}^n$ such that 
\begin{equation}\label{TechnicalExpansion}
	\forall t\in[0,\tau]: \quad P(t) = \sum_{k=0}^{N} \alpha_k\restr{\mathfrak{L}_k^\tau}{[0,\tau]}(t).
\end{equation}
Bearing in mind that $P\restr{\mathfrak{L}_k^\tau}{[0,\tau]} \in \restr{\mathbb{P}_{2N}^n}{[0,\tau]}$ for all $k = 0,1,\cdots,N$, we immediately conclude from \eqref{MutualOrthogonality} that 
\begin{equation*}
	\alpha_k = \frac{2k+1}{\tau} \int_{0}^{\tau}P(t)\mathfrak{L}_k^\tau(t)\dt, 
\end{equation*}
which, as a consequence of \eqref{IntegrationComputation}, can be written as 
\begin{equation}\label{TechnicalCoefficient}
	\alpha_k = \frac{2k+1}{\tau} \sum_{i = 0}^{N}P(t_{N,i})\mathfrak{L}_k^\tau(t_{N,i})w_{N,i}.
\end{equation} 
Thus, after defining the function $\phi_i: [0,\tau]\rightarrow\mathbb{R}$ as
\begin{equation}\label{PhiFunctionIntermediate}
	\phi_i(t) = \sum_{k=0}^{N}\frac{2k+1}{\tau}\big(\mathfrak{L}_k^\tau(t_{N,i})w_{N,i}\big)\mathfrak{L}_k^\tau(t),
\end{equation}
we substitute \eqref{TechnicalCoefficient} into \eqref{TechnicalExpansion} to obtain 
\begin{equation}\label{SpectralRepresentation}
	\forall t\in[0,\tau]: \quad P(t) = \sum_{i=0}^{N} \phi_i(t)P(t_{N,i}),
\end{equation}
which suggests that $P$ can be written in terms of its values at the quadrature nodes. More importantly, we may exploit \eqref{SpectralRepresentation} to obtain a \emph{closed-form} expression of \eqref{PhiFunctionIntermediate}. To accomplish this, we consider the \emph{scalar} function 
\begin{equation*}
	L_i(t) = \prod_{j=0,j\neq i}^{N}(t-t_{N,j}),
\end{equation*}
which is a polynomial, of degree $N$, whose roots include all quadrature nodes except for $t_{N,i}$. Bearing this in mind, we employ \eqref{SpectralRepresentation} to obtain $\phi_i(t) = L_i(t)L_i^{-1}(t_{N,i})$, from which we immediately conclude that
\begin{equation}\label{PhiFunction}
	\phi_i(t) = \prod_{j=0,j\neq i}^{N}\frac{t-t_{N,j}}{t_{N,i}-t_{N,j}}.
\end{equation}
We note from the expression \eqref{PhiFunction} that the function $\phi_i$ is in fact a \emph{Lagrange interpolating polynomial} for the quadrature nodes (see, \textit{e.g.,} \cite[Section 3.1]{faires2012numerical}). We further highlight that
\begin{equation}\label{Dirac}
	\phi_i(t_{N,j}) = \left\{\begin{aligned}
		&0&j&\neq i,\\
		&1&j&=i,
	\end{aligned} \right.
\end{equation}
for all $i=0,1,\ldots,N$. Additionally, given the special structure of \eqref{PhiFunction}, we may employ the product rule to obtain
\begin{equation}\label{DerPhiFunction}
	\dot{\phi}_i(t) = \sum_{l=0, l\neq i}^{N}\left(\frac{1}{t_{N,i}-t_{N,l}}\prod_{j=0, j\neq i,l}^{N} \frac{t-t_{N,j}}{t_{N,i}-t_{N,j}} \right).
\end{equation}

To recapitulate, we emphasize that any $P\in\restr{\mathbb{P}_N^n}{[0,\tau]}$ admits the representation \eqref{SpectralRepresentation}, with $\phi_i$ given by \eqref{PhiFunction}. We note that the function $\phi_i$ is \emph{independent} of $P$ as it is completely determined by the quadrature nodes. In the next section, we will extensively use the representation \eqref{SpectralRepresentation}, together with formulae \eqref{PhiFunction} and \eqref{DerPhiFunction}, to compare continuous-time signals with discrete-time ones. 
\section{System Interpolation}\label{Sec_SystemInterpolation}
We will use polynomial interpolation on the basis of shifted Legendre polynomials to propose a technique for comparison of the input-state behaviors of systems \eqref{ContinuousSystem} and \eqref{DiscreteSystem}. For this purpose, we first interpolate discrete-time signals into continuous-time \emph{piecewise polynomial} ones according to the following definition. 

\begin{definition}\label{DefSignalInterpolation}
	For $\ell,p\in\mathbb{Z}_{>0}$, let a discrete-time signal $s_d: \llbracket0,\ell \rrbracket \rightarrow \mathbb{R}^p$ be given. For a sampling time $\tau\in\mathbb{R}_{>0}$ and an integer $N\geq 1$, a continuous-time signal $s_c: [0,\ell \tau]\rightarrow\mathbb{R}^p$ is said to be an $N$-th order interpolation of $s_d$ with respect to the sampling time $\tau$ if there exist polynomial functions $P_0,P_1,\cdots,P_{\ell-1}\in\restr{\mathbb{P}_N^{p}}{[0,\tau]}$ such that for all $i \in \llbracket0,\ell-1\rrbracket$, we have that 
	\begin{subequations}\label{InterpolationCondition}
		\begin{equation}\label{InstanceEquivalence}
			\begin{aligned}
				s_c(i\tau) &= s_d(i),\\
				s_c\big((i+1)\tau\big) &= s_d(i+1)
			\end{aligned}
		\end{equation}
		\text{and}
		\begin{equation}\label{PolynomialRestriction}
			\forall t\in[0,\tau]: \quad s_c(i\tau+t) = P_i(t).
		\end{equation}
	\end{subequations}
\end{definition}
\vspace*{2mm}

According to Definition~\ref{DefSignalInterpolation}, the continuous-time signal $s_c: [0,\ell \tau]\rightarrow\mathbb{R}^p$ is an $N$-th order interpolation of the discrete-time signal $s_d: \llbracket0,\ell \rrbracket \rightarrow \mathbb{R}^p$ with respect to the sampling time $\tau$ if its value at any time $i\tau$ equals the value of $s_d$ at instance $i$, while its restriction to the interval $[i\tau,(i+1)\tau]$ can be described by a polynomial in $\restr{\mathbb{P}_N^{p}}{[0,\tau]}$. 

We note that for $N\geq2$, an $N$-th order interpolation of a discrete-time signal $s_d$ with respect to sampling time $\tau$ is \emph{not} unique. Thus, we denote by $\mathbb{I}_{N}^\tau(s_d)$ the set of all $N$-th order interpolations of a given discrete-time signal $s_d: \llbracket0,\ell \rrbracket \rightarrow \mathbb{R}^p$ with respect to the sampling time $\tau$. 


Having defined piecewise polynomial interpolation of discrete-time signals, we now propose the following notion of comparison that determines whether a continuous-time system is able to generate an \emph{interpolation} of the input-state behavior of a discrete-time system. 

\begin{definition}\label{DefInterpolatingSystem}
	For a sampling time $\tau\in\mathbb{R}_{>0}$ and an integer $N\geq 1$, a continuous-time system $\bm{\Sigma}_c$ is said to be an $N$-th order interpolator of a discrete-time system $\bm{\Sigma}_d$ with respect to the sampling time $\tau$ if for any integer $\ell\in\mathbb{Z}_{> 0}$, any initial condition $x_0\in\mathbb{R}^n$, and any discrete-time input $u_d:\llbracket 0,\ell \rrbracket \rightarrow \mathbb{R}^m$, there exists a continuous-time input $u_c \in \mathbb{I}_N^\tau (u_d)$ such that 
	\begin{equation}\label{InterpolatingCondition}
		\restr{x_c (\cdot; x_0,u_c)}{[0,\ell\tau]} \in \mathbb{I}_N^\tau \Big(\restr{x_d(\cdot;x_0,u_d)}{\llbracket 0,\ell\rrbracket}\Big).
	\end{equation}
\end{definition}
\vspace*{2mm}

It follows from Definition~\ref{DefInterpolatingSystem} that when $\bm{\Sigma}_c$ is an $N$-th order interpolator of $\bm{\Sigma}_d$ with respect to the sampling time $\tau$, for any input-state trajectory in $\bm{\Sigma}_d$, there exists an input-state trajectory in $\bm{\Sigma}_c$ that is its $N$-th order interpolation with respect to the sampling time $\tau$. In other words, $\bm{\Sigma}_c$ is able to generate an $N$-th order interpolation of any input-state trajectory of $\bm{\Sigma}_d$ with respect to the sampling time $\tau$. 

We now propose the following result which allows for the investigation of system interpolation on the basis of a \emph{single} (discrete) time step, \textit{i.e.,} in order to inspect system interpolation, it is sufficient to consider a single interval $[0,\tau]$. This will later facilitate an algebraic characterization of system interpolation. 

\begin{proposition}\label{PropositionSingleStepCheck}
	For a sampling time $\tau\in\mathbb{R}_{>0}$ and an integer $N\geq 1$, $\bm{\Sigma}_c$ is an $N$-th order interpolator of $\bm{\Sigma}_d$ with respect to the sampling time $\tau$ if and only if for any $x_0\in\mathbb{R}^n$ and any $u_d: \llbracket0,1\rrbracket\rightarrow\mathbb{R}^m$, there exists $u_c \in \mathbb{I}_N^\tau(u_d)$ such that 
	\begin{equation}\label{InterpolatingConditionSimple}
		\restr{x_c (\cdot; x_0,u_c)}{[0,\tau]} \in \mathbb{I}_N^\tau \Big(\restr{x_d(\cdot;x_0,u_d)}{\llbracket 0,1\rrbracket}\Big).
	\end{equation}
\end{proposition}
\vspace*{2mm}
\begin{proof}
	To show necessity, let $\ell = 1$, $x_0\in\mathbb{R}^n$, and $u_d:\llbracket 0,\ell \rrbracket \rightarrow \mathbb{R}^m$. As $\bm{\Sigma}_c$ is an $N$-th order interpolator of $\bm{\Sigma}_d$ with respect to the sampling time $\tau$, it follows from Definition~\ref{DefInterpolatingSystem} that there exists $u_c \in \mathbb{I}_N^\tau(u_d)$ such that \eqref{InterpolatingConditionSimple} holds. 
	
	To show sufficiency, let $\ell\in\mathbb{Z}_{>0}$, $x_0\in\mathbb{R}^n$, and $u_d: \llbracket0,\ell\rrbracket\rightarrow\mathbb{R}^m$. For any integer $0\leq i\leq \ell-1$, we define $u_{d}^i: \llbracket 0,1\rrbracket\rightarrow\mathbb{R}^m$ such that 
	\begin{equation}\label{PropTechnical1}
		u_{d}^i(0) = u_d(i), \quad u_d^i(1) = u_d(i+1).
	\end{equation}
	We then note from time-invariance of $\bm{\Sigma}_d$ that 
	\begin{equation}\label{DiscreteInvariance}
		\restr{x_d\big(\cdot; x_d(i; x_0,u_d),u_d^i \big)}{\llbracket0,1\rrbracket} = \restr{x_d(\cdot;x_0,u_d)}{\llbracket i,i+1\rrbracket}.
	\end{equation}
	It then follows from the proposition statement that there exists $u_c^i\in\mathbb{I}_N^\tau(u_d^i)$ such that 
	\begin{equation*}
	{\textstyle	\restr{x_c \big(\cdot; x_d(i; x_0,u_d),u_c^i\big)}{[0,\tau]}\!\in\!\mathbb{I}_N^\tau \Big(\! \restr{x_d\big(\cdot;x_d(i; x_0,u_d),u_d^i\big)}{\llbracket0,1\rrbracket}\!\Big)}.
	\end{equation*}
	Then, bearing \eqref{DiscreteInvariance} in mind, we conclude from Definition~\ref{DefSignalInterpolation} that there exists $P_{i}\in\restr{\mathbb{P}_N^n}{[0,\tau]}$ such that
	\begin{subequations}\label{PropTechnical2}
		\begin{equation}\label{PropTechnical21}
			\begin{aligned}
				x_c\big(0; x_d(i; x_0,u_d),u_c^i\big)&= x_d(i;x_0,u_d),\\
				x_c\big(\tau; x_d(i; x_0,u_d),u_c^i\big) &= x_d(i+1;x_0,u_d),\\
			\end{aligned}
		\end{equation}
		\text{and}
		\begin{equation}\label{PropTechnical31}
			\begin{aligned}
				\forall t\in[0,\tau]:x_c\big(t; x_d(i; x_0,u_d),u_c^i\big) &= P_{i}(t).
			\end{aligned}
		\end{equation}
	\end{subequations}
	We now construct the continuous-time input $u_c: [0,\ell \tau]\rightarrow\mathbb{R}^m$ such that for all integers $0\leq i \leq \ell-1$,
	\begin{equation}\label{PropTechnical4}
		\forall t\in[0,\tau]: \quad u_c(i\tau+t) = u_c^i(t).
	\end{equation}
	Recalling that $u_c^i\in\mathbb{I}_N^\tau(u_d^i)$, we utilize \eqref{PropTechnical4}, together with \eqref{PropTechnical1}, to conclude that $u_c \in\mathbb{I}_N^\tau(u_d)$. We now show that \eqref{InterpolatingCondition} holds. For this purpose, we utilize induction to show that  for all $i\in\llbracket0,\ell-1 \rrbracket$, 
	\begin{subequations}\label{PropTechnical5}
		\begin{equation}\label{PropTechnical51}
			\begin{aligned}
				x_c (i\tau; x_0,u_c) &= x_d(i;x_0,u_d),\\
				x_c \big((i+1)\tau; x_0,u_c\big) &= x_d(i+1;x_0,u_d),
			\end{aligned}
		\end{equation}
		\text{and}
		\begin{equation}\label{PropTechnical52}
			\forall t\in[0,\tau]: \quad x_c (i\tau+t; x_0,u_c) = P_i(t).
		\end{equation}
	\end{subequations}
	Bearing \eqref{PropTechnical4} in mind, we note from the uniqueness of the solution of \eqref{ContinuousSystem} that 
	\begin{equation*}
	\begin{aligned}
			\forall t\in [0,\tau]: \quad x_c (t; x_0,u_c) &=  x_c \big(t; x_0,u_{c}^0\big)\\
			&= x_c \big(t; x_d(0; x_0,u_d),u_{c}^0\big),
	\end{aligned}
	\end{equation*}
	which, by \eqref{PropTechnical2}, implies that \eqref{PropTechnical5} holds for $i=0$. We now suppose that \eqref{PropTechnical5} holds for $i=k-1$, where $0<k\leq \ell-1$. It then follows from time-invariance of $\bm{\Sigma}_c$ that for all $t\in[0,\tau]$,
	\begin{equation*}
		x_c (k\tau+t; x_0,u_c) = x_c \big(t; x_c(k\tau; x_0,u_c),u_c^k\big),
	\end{equation*} 
	which, as a result of \eqref{PropTechnical51}, can be written as 
	\begin{equation*}
		\begin{aligned}
			x_c (k\tau+t; x_0,u_c) = x_c \big(t; x_d(k; x_0,u_d),u_c^k\big).
		\end{aligned}
	\end{equation*}
	It then follows from \eqref{PropTechnical2} that 
	\begin{equation*}
		\begin{aligned}
			x_c (k\tau; x_0,u_c) &= x_d(k;x_0,u_d),\\
			x_c\big((k+1)\tau; x_0,u_c\big) &= x_d(k+1;x_0,u_d),
		\end{aligned}
	\end{equation*}
	and 
	\begin{equation*}
		\forall t\in[0,\tau]: \quad x_c (k\tau+t; x_0,u_c) = P_k(t),
	\end{equation*}
	indicating that \eqref{PropTechnical5} holds for all $i\in\llbracket0,\ell-1 \rrbracket$. This, by Definition~\ref{DefSignalInterpolation}, implies \eqref{InterpolatingCondition}. 
\end{proof}

We now utilize Proposition~\ref{PropositionSingleStepCheck} to obtain an algebraic characterization of system interpolation. To accomplish this, for a sampling time $\tau\in\mathbb{R}_{>0}$, an integer $N\geq 1$, an initial condition $x_0\in\mathbb{R}^n$, and a discrete-time input $u_d:\llbracket0,1\rrbracket\rightarrow\mathbb{R}^m$, we suppose there exists a continuous-time input $u_c\in\mathbb{I}_N^\tau(u_d)$ such that \eqref{InterpolatingConditionSimple} holds. 

Considering that $u_c\in\mathbb{I}_N^\tau(u_d)$, we conclude from Definition~\ref{DefSignalInterpolation} that
\begin{equation}\label{UBoundaryCondition}
	u_c(0) = u_d(0), \quad u_c(\tau) = u_d(1),
\end{equation}
and that $u_c \in \restr{\mathbb{P}_N^{n}}{[0,\tau]}$. Accordingly, recalling that $t_{N,0} = 0$, we use \eqref{SpectralRepresentation} and \eqref{UBoundaryCondition} to write
\begin{equation}\label{USpectral}
	\forall t\in[0,\tau]: \quad u_c(t) = \phi_0(t)u_d(0) + \sum_{i=1}^{N}\phi_i(t)u_c(t_{N,i}),
\end{equation}
where, for $i=0,1,\ldots,N$, the function $\phi_i$ is defined as in \eqref{PhiFunction} and is therefore independent from $u_d$. 

We accordingly define the vector 
\begin{equation}\label{Phi}
	\Phi = \begin{bmatrix}
		\phi_1(\tau)\\\phi_2(\tau)\\\vdots\\\phi_N(\tau)
	\end{bmatrix}
\end{equation}
and the matrix 
\begin{equation}\label{U}
	U=\begin{bmatrix}
		u_c(t_{N,1})&u_c(t_{N,2})&\cdots&u_c(t_{N,N})
	\end{bmatrix},
\end{equation}
which, together with \eqref{USpectral}, allows us to write \eqref{UBoundaryCondition} as
\begin{equation}\label{UFinalCondition}
	U\Phi = u_d(1) - \phi_0(\tau)u_d(0).
\end{equation}

We now consider $x_c(\cdot;x_0,u_c)$ and utilize \eqref{InterpolatingConditionSimple} to conclude from Definition~\ref{DefSignalInterpolation} that 
\begin{equation}\label{XBoundaryCondition}
	\begin{aligned}
		x_c(0;x_0,u_c) &= x_0, 
		&x_c(\tau;x_0,u_c) &= x_d(1;x_0,u_d),  
	\end{aligned}
\end{equation}
and that $\restr{x_c(\cdot;x_0,u_c)}{[0,\tau]} \in \restr{\mathbb{P}_N^{n}}{[0,\tau]}$. This, using again \eqref{SpectralRepresentation}, implies that for all $t\in[0,\tau]$,
\begin{equation}\label{XSpectral}
	\begin{aligned}
		x_c(t;x_0,u_c) = \phi_0(t)x_0+ \sum_{i=1}^{N} \phi_i(t)x_c(t_{N,i};x_0,u_c).
	\end{aligned}
\end{equation}
Moreover, for $i=1,2,\ldots,N$, we define 
\begin{equation}\label{Xi}
	X_i = x_c(t_{N,i};x_0,u_c)
\end{equation}
and construct the matrix 
\begin{equation}\label{XMatrix}
	\begin{aligned}
		X &= \begin{bmatrix}
			X_1&X_2&\cdots&X_N
		\end{bmatrix},
	\end{aligned}
\end{equation}
accordingly. We then exploit \eqref{XSpectral} to write \eqref{XBoundaryCondition} as
\begin{equation}\label{XFinalCondition}
	X\Phi = x_d(1;x_0,u_d) - \phi_0(\tau)x_0,
\end{equation}
where $\Phi$ is defined as in \eqref{Phi}. Furthermore, after differentiating \eqref{XSpectral} with respect to $t$, we substitute $t = t_{N,j}$ and conclude from \eqref{ContinuousSystem} that 
\begin{align}
		\sum_{i=1}^{N} \dot{\phi}_i(t_{N,j})x_c(t_{N,i};x_0,u_c)&= A_c x_c(t_{N,j};x_0,u_c)\label{DerivativeEquation}\\
		&\quad+ B_cu_c(t_{N,j})-\dot{\phi}_0(t_{N,j})x_0,\nonumber
\end{align}
for all $j=0,1,\cdots,N$. We recall that the derivative $\dot{\phi}_i(t)$ is computed by \eqref{DerPhiFunction}. Subsequently, after defining the vector
\begin{equation}\label{PsiVector}
	\psi_j = \begin{bmatrix}
		\dot{\phi}_1(t_{N,j})\\\dot{\phi}_2(t_{N,j})\\\vdots\\\dot{\phi}_{N}(t_{N,j})\\
	\end{bmatrix},
\end{equation}
for $j=0,1,\cdots,N$, we construct the matrix 
\begin{equation}\label{PhiMatrix}
	\Psi = \begin{bmatrix}
		\psi_1&\psi_2&\cdots&\psi_{N}
	\end{bmatrix}.
\end{equation}
We write \eqref{DerivativeEquation}, evaluated at $j=0$, as 
\begin{subequations}\label{XDerivativeCondition}
	\begin{equation}
		X\psi_0 = \big(A_c-\dot{\phi}_0(0)I_n\big)x_0 + B_cu_d(0),
	\end{equation}
	\text{whereas we collect \eqref{DerivativeEquation}, evaluated at $j=1,2,\ldots,N$, into
	}
	\begin{equation}
		X\Psi = A_cX+B_cU - x_0\sigma,
	\end{equation}
\end{subequations}
where
\begin{equation}\label{SigmaVectors}
	\sigma = \begin{bmatrix}
		\dot{\phi}_0(t_{N,1})&\cdots&\dot{\phi}_0(t_{N,N})
	\end{bmatrix},
\end{equation}
Finally, bearing in mind that 
\begin{equation*}
	x_d(1;x_0,u_d) = A_d x_0 + B_du_d(0),
\end{equation*}
we collect \eqref{UFinalCondition}, \eqref{XFinalCondition}, and \eqref{XDerivativeCondition} to obtain 
\begin{equation}\label{MidEquations}
	\begin{aligned}
		X\psi_0 &= \big(A_c-\dot{\phi}_0(0)I_n\big)x_0 + B_cu_d(0),\\
		X\Psi - A_cX - B_cU&= - x_0\sigma,\\
		X\Phi &= \big(A_d - \phi_0(\tau)I_n\big)x_0 + B_du_d(0),\\
		U\Phi &= u_d(1) - \phi_0(\tau)u_d(0).
	\end{aligned}
\end{equation}
The formulation \eqref{MidEquations} facilitates the following algebraic characterization of system interpolation, which is solely in terms of the parameters of $\bm{\Sigma}_c$ and $\bm{\Sigma}_d$. \begin{theorem}\label{ThCharacterization}
	For a sampling time $\tau\in\mathbb{R}_{>0}$ and an integer $N\geq 1$, $\bm{\Sigma}_c$ is an $N$-th order interpolator of $\bm{\Sigma}_d$ with respect to the sampling time $\tau$ if and only if 
	\begin{equation}\label{Finalcharacterization}
		\begin{aligned}
			\operatorname{Im}\!\!\begingroup 
			\setlength\arraycolsep{1.25pt}\begin{bmatrix}
				A_c-\dot{\phi}_0(0)I&B_c&0\\
				-\sigma^\top\otimes I&0&0\\
				A_d - \phi_0(\tau)I&B_d&0\\
				0&\!\!\!-\phi_0(\tau)I&I
			\end{bmatrix}\endgroup \!\subset
			\operatorname{Im}\!\!\begingroup 
			\setlength\arraycolsep{2pt}\begin{bmatrix}
				\psi_0^\top \otimes I & 0\\
				\Psi^\top\oplus(-A_c)&-I \otimes B_c\\
				\Phi^\top \otimes I&0\\
				0 &\Phi^\top \otimes I
			\end{bmatrix} 
			\endgroup\!,
		\end{aligned}
	\end{equation}
	where $\phi_0$, $\Phi$, $\psi_0$, $\Psi$, and $\sigma$  are respectively given by \eqref{PhiFunction}, \eqref{Phi}, \eqref{PsiVector}, \eqref{PhiMatrix}, and \eqref{SigmaVectors}.  
\end{theorem}
\vspace*{2mm}
\begin{proof}
	To show necessity, suppose $\bm{\Sigma}_c$ is an $N$-th order interpolator of $\bm{\Sigma}_d$ with respect to the sampling time $\tau$. Let $x_0\in\mathbb{R}^n$ and $u_d:\llbracket0,1\rrbracket\rightarrow\mathbb{R}^m$. It then follows from Proposition~\ref{PropositionSingleStepCheck} that there exists $u_c\in\mathbb{I}_N^\tau (u_d)$ such that \eqref{InterpolatingConditionSimple} holds. We then follow the procedure described above to conclude that \eqref{MidEquations} holds for matrices $U$ and $X$ respectively defined as in \eqref{U} and \eqref{XMatrix}. This, as a consequence of matrix vectorization, results in 
	\begin{equation}\label{ThTechnical1}
		\begin{aligned}
			&\begingroup 
			\setlength\arraycolsep{2pt}\begin{bmatrix}
				\psi_0^\top \otimes I & 0\\
				\Psi^\top\oplus(-A_c)&-I \otimes B_c\\
				\Phi^\top \otimes I&0\\
				0 &\Phi^\top \otimes I
			\end{bmatrix} 
			\endgroup \begin{bmatrix}
				\operatorname{vec}(X)\\\operatorname{vec}(U)
			\end{bmatrix} 
			\\
			&= \begingroup 
			\setlength\arraycolsep{2pt}\begin{bmatrix}
				A_c-\dot{\phi}_0(0)I&B_c&0\\
				-\sigma^\top\otimes I&0&0\\
				A_d - \phi_0(\tau)I&B_d&0\\
				0&-\phi_0(\tau)I&I
			\end{bmatrix}\endgroup \begin{bmatrix}
				x_0\\u_d(0)\\u_d(1)
			\end{bmatrix}.
		\end{aligned}
	\end{equation}
	As $x_0$ and $u_d$ are chosen arbitrarily, we conclude from \eqref{ThTechnical1} that \eqref{Finalcharacterization} holds. 
	
	To show sufficiency, suppose that \eqref{Finalcharacterization} holds. Letting $x_0\in\mathbb{R}^n$ and $u_d:\llbracket0,1\rrbracket\rightarrow\mathbb{R}^m$, we conclude from \eqref{Finalcharacterization} that there exist matrices $X\in\mathbb{R}^{n\times N}$ and $U\in\mathbb{R}^{m\times N}$ such that \eqref{ThTechnical1} holds. We will use this matrix $U$ to construct $u_c:[0,\tau]\rightarrow\mathbb{R}^m$ such that $u_c\in\mathbb{I}_N^\tau(u_d)$ and that \eqref{InterpolatingConditionSimple} holds. We then utilize Proposition~\ref{PropositionSingleStepCheck} to conclude that $\bm{\Sigma}_c$ is an $N$-th order interpolator of $\bm{\Sigma}_d$ with respect to the sampling time $\tau$.
	
	As the first step, we utilize matrix vectorization to equivalently write \eqref{ThTechnical1} as \eqref{MidEquations}. We let  $U_1,U_2,\cdots,U_N\in\mathbb{R}^m$ be the columns of $U$, \textit{i.e.,} $U = [U_1 \; U_2 \; \cdots \; U_N]$. We then define $u_c:[0,\tau]\rightarrow\mathbb{R}^{m}$ such that
	\begin{equation}
		\forall t\in[0,\tau]: \quad u_c(t) = \phi_0(t)u_d(0) + \sum_{i=1}^{N}\phi_i(t)U_i,
	\end{equation}
	which, as a consequence of \eqref{PhiFunction}, implies that $u_c\in\restr{\mathbb{P}_N^m}{[0,\tau]}$. We note from \eqref{Dirac} that $u_c(t_{N,0}) = u_d(0)$ and that
	\begin{equation*}
		\forall i\in\{1,\ldots,N\}:\quad u_c(t_{N,i}) = U_i.
	\end{equation*}
	We further note from the last equation of \eqref{MidEquations} that $u_c(\tau) = u_d(1)$. This, by Definition~\ref{DefSignalInterpolation}, indicates that $u_c\in\mathbb{I}_N^\tau(u_d)$. 
	
	We now show that \eqref{InterpolatingConditionSimple} holds. For this purpose, we let $X_1,X_2,\cdots,X_N\in\mathbb{R}^n$ be such that \eqref{XMatrix} holds. We accordingly define $\bar{x}_c:[0,\tau]\rightarrow\mathbb{R}^n$ such that 
	\begin{equation*}
		\bar{x}_c(t) = \phi_0(t) x_0 + \sum_{i=1}^{N}\phi_i(t)X_i.
	\end{equation*}
	which, as a result of \eqref{PhiFunction}, implies that $\bar{x}_c\in\restr{\mathbb{P}_N^n}{[0,\tau]}$. Furthermore, it follows from \eqref{Dirac} that $\bar{x}_c(t_{N,0}) = x_0$ and that
	\begin{equation*}
		\forall i\in\{1,\ldots,N\}:\quad \bar{x}_c(t_{N,i}) = X_i,
	\end{equation*}
	which, as a consequence of the third equation of \eqref{MidEquations}, implies that $\bar{x}_c(\tau) = x_d(1;x_0,u_d)$. This, by Definition~\ref{DefSignalInterpolation}, indicates that
	\begin{equation}\label{ThTechnical5}
		\bar{x}_c \in \mathbb{I}_N^\tau \Big(\restr{x_d(\cdot;x_0,u_d)}{\llbracket0,1\rrbracket}\Big).
	\end{equation}
	It now remains to show that $\restr{x_c(\cdot;x_0,u_c)}{[0,\tau]} = \bar{x}_c$. To accomplish this, we first utilize the first and second equations of \eqref{MidEquations} to conclude that 
	\begin{equation*}
		\forall i\in\{0,1,\cdots,N\}: \quad \dot{\bar{x}}_c (t_{N,i}) = A_c\bar{x}_c(t_{N,i}) + B_cu_c(t_{N,i}).
	\end{equation*}
	Bearing in mind that $\bar{x}_c, \dot{\bar{x}}_c\in\restr{\mathbb{P}_N^{n}}{[0,\tau]}$ and $u_c\in\restr{\mathbb{P}_N^m}{[0,\tau]}$, we utilize \eqref{NormComputation} to obtain 
	\begin{equation*}
		\begin{aligned}
			&\int_{0}^{\tau} |\dot{\bar{x}}_c(t)-A_c\bar{x}_c(t)-B_cu_c(t)|^2\dt\\
			 &= \sum_{i=0}^{N}|\dot{\bar{x}}_c (t_{N,i}) - A_c\bar{x}_c(t_{N,i}) - B_cu_c(t_{N,i})|^2w_{N,i}\\
			&= 0,
		\end{aligned}
	\end{equation*}
	which, in turn, implies that for all $t\in[0,\tau]$,
	\begin{equation*}
		\dot{\bar{x}}_c(t) = A_c\bar{x}_c(t)+B_cu_c(t). 
	\end{equation*}
	We now define $\tilde{x}_c: [0,\tau]\rightarrow\mathbb{R}^n$ as
	\begin{equation*}
		\forall t\in[0,\tau]: \quad  \tilde{x}_c(t) = x_c(t;x_0,u_c) - \bar{x}_c(t).
	\end{equation*}
	We note from \eqref{ThTechnical5} that $\tilde{x}_c(0) = 0$. We then utilize \eqref{ContinuousSystem} to obtain
	\begin{equation}\label{Technical1}
		\forall t\in[0,\tau]: \quad \dot{\tilde{x}}_c(t) = A_c \tilde{x}_c(t).
	\end{equation}
	By recalling that $\tilde{x}_c(0) = 0$, we conclude from uniqueness of the solution of \eqref{Technical1} that
	\begin{equation*}
		\forall t\in[0,\tau]:\quad \tilde{x}_c(t) = 0,
	\end{equation*}
	which indicates that $\restr{x_c(\cdot;x_0,u_c)}{[0,\tau]} = \bar{x}_c$. It then follows from \eqref{ThTechnical5} that \eqref{InterpolatingConditionSimple} holds. As $x_0\in\mathbb{R}^n$ and $u_d:\llbracket0,1\rrbracket\rightarrow\mathbb{R}^m$ are chosen arbitrarily, we conclude from Proposition~\ref{PropositionSingleStepCheck} that $\bm{\Sigma}_c$ is an $N$-th order interpolator of $\bm{\Sigma}_d$ with respect to $\tau$.
\end{proof}

Theorem~\ref{ThCharacterization} allows one to determine if $\bm{\Sigma}_c$ is an $N$-th order interpolator of $\bm{\Sigma}_d$ with respect to the sampling time $\tau$ by simply checking the subspace inclusion \eqref{Finalcharacterization}, which is solely in terms of $\tau$, $N$, and the parameters of $\bm{\Sigma}_c$ and $\bm{\Sigma}_d$. This is computationally efficient as \eqref{Finalcharacterization} can be equivalently formulated as a \emph{rank} condition. 

Having characterized system interpolation, we now suppose that $\bm{\Sigma}_c$ is an $N$-th order interpolator of $\bm{\Sigma}_d$ with respect to the sampling time $\tau$. Given $\ell\in\mathbb{Z}_{>0}$, $x_0\in\mathbb{R}^n$, and $u_d:\llbracket 0,\ell\rrbracket\rightarrow\mathbb{R}^m$, we employ Theorem~\ref{ThCharacterization} and Proposition~\ref{PropositionSingleStepCheck} to characterize the set of all continuous-time inputs $u_c\in\mathbb{I}_N^\tau(u_d)$ that establish \eqref{InterpolatingCondition}. To accomplish this, for any integer $0\leq i\leq \ell-1$, we let $X^i\in\mathbb{R}^{n\times N}$ and $U^i\in\mathbb{R}^{m\times N}$ be such that
\begin{equation}\label{AlgorithmTechnical1}
	\begin{aligned}
		&\begingroup 
		\setlength\arraycolsep{2pt}\begin{bmatrix}
			\psi_0^\top \otimes I & 0\\
			\Psi^\top\oplus (-A_c)&-I \otimes B_c\\
			\Phi^\top \otimes I&0\\
			0 &\Phi^\top \otimes I
		\end{bmatrix} 
		\endgroup \begin{bmatrix}
			\operatorname{vec}(X^i)\\\operatorname{vec}(U^i)
		\end{bmatrix}\\
		& = \begingroup 
		\setlength\arraycolsep{2pt}\begin{bmatrix}
			A_c-\dot{\phi}_0(0)I&B_c&0\\
			-\sigma^\top\otimes I&0&0\\
			A_d - \phi_0(\tau)I&B_d&0\\
			0&-\phi_0(\tau)I&I
		\end{bmatrix}\endgroup \begin{bmatrix}
			x_d(i;x_0,u_d)\\u_d(i)\\u_d(i+1)
		\end{bmatrix}.
	\end{aligned}
\end{equation}
We note from Theorem~\ref{ThCharacterization} that \eqref{AlgorithmTechnical1} is always \emph{consistent}, \textit{i.e.,} there exist matrices $X^i$ and $U^i$ such that \eqref{AlgorithmTechnical1} holds. It is also worth noting that the matrices $X^i$ and $U^i$ are \emph{not} necessarily unique. We then partition $U^i$ as $U^i = [U_1^i\,U_2^i\,\cdots\,U_N^i]$ and accordingly define $u_c^i\in\restr{\mathbb{P}_N^m}{[0,\tau]}$ such that 
\begin{equation}\label{AlgorithmTechnical2}
	\forall t\in[0,\tau]:\quad u_c^i(t) = \phi_0(t) u_d(i) + \sum_{j=1}^{N}\phi_j(t)U_j^i.
\end{equation}
Hence, we observe that the solutions of \eqref{AlgorithmTechnical1} induce a set
\begin{equation}\label{SetUci}
	\mathcal{U}_c^i(x_0,u_d)\! = \!\left\{\left. \!u_c^i\in\restr{\mathbb{P}_N^m}{[0,\tau]} \right\vert \!\text{$\exists X^i\!,U^i$s.t. \eqref{AlgorithmTechnical1}--\eqref{AlgorithmTechnical2} hold} \right\}\!,
\end{equation}
based on which, we construct the set of continuous-time inputs 
\begin{align}
	\mathcal{U}_c(x_0,u_d) = \Big\{u_c: [0,\ell \tau]&\rightarrow\mathbb{R}^m \big\vert \text{$\forall i\in\llbracket0,\ell-1\rrbracket,$}\label{SetUc}\\
	&\text{$\exists u_c^i\in\mathcal{U}_c^i(x_0,u_d)$ s.t. \eqref{PropTechnical4} holds}\Big\} .\nonumber
\end{align}
We note that since \eqref{AlgorithmTechnical1} is consistent for any choice of $\ell\in\mathbb{Z}_{>0}$, $x_0\in\mathbb{R}^n$, and $u_d:\llbracket0,\ell\rrbracket \rightarrow\mathbb{R}^m$, the set $\mathcal{U}_c(x_0,u_d)$ is always \emph{non-empty}. This observation is formally stated in the following lemma. 
\begin{lemma}\label{LmNonEmpty}
	For a sampling time $\tau\in\mathbb{R}_{>0}$ and an integer $N\geq 1$, suppose $\bm{\Sigma}_c$ is an $N$-th order interpolator of $\bm{\Sigma}_d$ with respect to the sampling time $\tau$. Then, for any $\ell\in\mathbb{Z}_{>0}$, $x_0\in\mathbb{R}^n$, and $u_d:\llbracket0,\ell\rrbracket \rightarrow\mathbb{R}^m$, the set of continuous-time inputs $\mathcal{U}_c(x_0,u_d)$, defined in \eqref{SetUc}, is non-empty. 
\end{lemma}
\begin{proof}
	Let $\ell\in\mathbb{Z}_{>0}$, $x_0\in\mathbb{R}^n$, and $u_d:\llbracket0,\ell\rrbracket \rightarrow\mathbb{R}^m$. Since $\bm{\Sigma}_c$ is an $N$-th order interpolator of $\bm{\Sigma}_d$ with respect to $\tau$, it follows from Theorem~\ref{ThCharacterization} that for any integer $0\leq i\leq \ell-1$, there exist matrices $X^i\in\mathbb{R}^{n\times N}$ and $U^i\in\mathbb{R}^{m\times N}$ such that \eqref{AlgorithmTechnical1} holds. We may therefore use this matrix $U^i$ to define $u_c^i\in\restr{\mathbb{P}_N^m}{[0,\tau]}$ as in \eqref{AlgorithmTechnical2}. It thus follows from \eqref{SetUci} that $\mathcal{U}_c^i(x_0,u_d)$ is non-empty for all $0\leq i\leq \ell-1$, which, by \eqref{SetUc}, implies that $\mathcal{U}_c(x_0,u_d)$ is non-empty.
\end{proof}

We now exploit Lemma~\ref{LmNonEmpty} to obtain the following result, which shows that $\mathcal{U}_c(x_0,u_d)$ is the set of all continuous-time inputs $u_c\in\mathbb{I}_N^\tau(u_d)$ that establish \eqref{InterpolatingCondition}.
\begin{theorem}\label{ThInputSetCharacterization}
	For a sampling time $\tau\in\mathbb{R}_{>0}$ and an integer $N\geq 1$, suppose $\bm{\Sigma}_c$ is an $N$-th order interpolator of $\bm{\Sigma}_d$ with respect to the sampling time $\tau$. Then, for any integer $\ell \in \mathbb{Z}_{>0}$, any initial condition $x_0\in\mathbb{R}^n$, and any discrete-time input $u_d:\llbracket 0,\ell \rrbracket \rightarrow \mathbb{R}^m$, 
	\begin{equation}\label{CompleteSet}
		u_c \in \mathcal{U}_c(x_0,u_d) \Longleftrightarrow  \text{$u_c\in\mathbb{I}_N^\tau (u_d)$ s.t. \eqref{InterpolatingCondition} holds,}
	\end{equation}
	where $\mathcal{U}_c(x_0,u_d)$ is defined as in \eqref{SetUc}.
\end{theorem}
\begin{proof}
	Let $\ell \in \mathbb{Z}_{>0}$, $x_0\in\mathbb{R}^n$, and $u_d:\llbracket 0,\ell \rrbracket \rightarrow \mathbb{R}^m$. Since $\bm{\Sigma}_c$ is an $N$-th order interpolator of $\bm{\Sigma}_d$ with respect to $\tau$, it follows from Lemma~\ref{LmNonEmpty} that $\mathcal{U}_c(x_0,u_d)$ is non-empty. 
	
	We first let $u_c\in\mathcal{U}_c(x_0,u_d)$ and show that $u_c\in\mathbb{I}_N^\tau(u_d)$ and that \eqref{InterpolatingCondition} holds. To do so, for any integer $0\leq i\leq \ell-1$, we define $u_{d}^i: \llbracket 0,1\rrbracket\rightarrow\mathbb{R}^m$ according to \eqref{PropTechnical1}. We also let $u_c^i:[0,\tau]\rightarrow\mathbb{R}^m$ be such that \eqref{PropTechnical4} holds. It then follows from the definition of $\mathcal{U}_c(x_0,u_d)$ in \eqref{SetUci} and \eqref{SetUc} that there exist $X^i\in\mathbb{R}^{n\times N}$ and $U^i\in\mathbb{R}^{m\times N}$ such that \eqref{AlgorithmTechnical1} and \eqref{AlgorithmTechnical2} hold. Bearing \eqref{PropTechnical1} in mind, we follow the \emph{sufficiency} part of the proof of Theorem~\ref{ThCharacterization} to conclude that $u_c^i\in\mathbb{I}_N^\tau(u_d^i)$.  Furthermore, after partitioning $X^i$ as $X^i = [X_1^i \, X_2^i\,\cdots\,X_N^i]$, the procedure described in the proof of Theorem~\ref{ThCharacterization} gives
	\begin{equation*}
		{\textstyle	\restr{x_c \big(\cdot; x_d(i; x_0,u_d),u_c^i\big)}{[0,\tau]}\!\in\!\mathbb{I}_N^\tau \Big(\! \restr{x_d\big(\cdot;x_d(i; x_0,u_d),u_d^i\big)}{\llbracket0,1\rrbracket}\!\Big)}.
	\end{equation*}
	Additionally, we have that 
	\begin{subequations}\label{Technical5}
		\begin{equation}
			\restr{x_c \big(0; x_d(i; x_0,u_d),u_c^i\big)}{[0,\tau]} = x_d(i; x_0,u_d),
		\end{equation}
		\text{and that, at the quadrature nodes,}
		\begin{equation}
			\begin{aligned}
				\restr{x_c \big(t_{N,j}; x_d(i; x_0,u_d),u_c^i\big)}{[0,\tau]} = X_j^i,
			\end{aligned}
		\end{equation}
	\end{subequations}
	for all $j=1,2,\ldots,N$.
	We then recall \eqref{DiscreteInvariance} and follow the \emph{sufficiency} part of the proof of Proposition~\ref{PropositionSingleStepCheck} to conclude that $u_c\in\mathbb{I}_N^\tau(u_d)$ and that \eqref{InterpolatingCondition} holds. Importantly, it also follows from the proof of Proposition~\ref{PropositionSingleStepCheck} that for all integers $0\leq i\leq \ell-1$, we have that for all $ t\in[0,\tau]$,
	\begin{equation}\label{AlgorithmTechnical3}
		x_c (i\tau+t; x_0,u_c) = \phi_0(t) x_d(i;x_0,u_d) + \sum_{j=1}^{N}\phi_j(t)X_j^i.
	\end{equation}
	Moreover, it follows from \eqref{Technical5} that 
	\begin{subequations}\label{Technical6}
		\begin{equation}
			x_c (i\tau; x_0,u_c) = x_d(i;x_0,u_d),
		\end{equation}
		\text{and that }
		\begin{equation}
			\forall j\in\{1,2,\ldots,N\}: \quad x_c (i\tau+t_{N,j}; x_0,u_c) = X_j^i.
		\end{equation}
	\end{subequations}
	
	We now let $u_c\in\mathbb{I}_N^\tau(u_d)$ be such that \eqref{InterpolatingCondition} holds, and we show that $u_c\in\mathcal{U}_c(x_0,u_d)$. To accomplish this, for any integer $0\leq i\leq \ell-1$, we define $u_{d}^i: \llbracket 0,1\rrbracket\rightarrow\mathbb{R}^m$ and $u_c^i:[0,\tau]\rightarrow\mathbb{R}^m$ such that \eqref{PropTechnical1} and \eqref{PropTechnical4} hold. Bearing \eqref{DiscreteInvariance} in mind, we conclude from $u_c\in\mathbb{I}_N^\tau(u_d)$ and \eqref{InterpolatingCondition} that $u_c^i\in\mathbb{I}_N^\tau(u_d^i)$ and that
	\begin{equation*}
		{\textstyle	\restr{x_c \big(\cdot; x_d(i; x_0,u_d),u_c^i\big)}{[0,\tau]}\!\in\!\mathbb{I}_N^\tau \Big(\! \restr{x_d\big(\cdot;x_d(i; x_0,u_d),u_d^i\big)}{\llbracket0,1\rrbracket}\!\Big)}.
	\end{equation*}
	Now, for $j=1,2,\ldots,N$, we define
	\begin{equation*}
		\begin{aligned}
			U_j^i &= u_c^i(t_{N,j}), & X_j^i &= x_c(i\tau+t_{N,j};x_0,u_c),
		\end{aligned}
	\end{equation*} 
	and accordingly construct the matrices 
	\begin{equation}\label{Technical3}
		\begin{aligned}
			U^i=\begin{bmatrix}
				U_1^i, U_2^i,\cdots,U_N^i
			\end{bmatrix},\;\; X^i = \begin{bmatrix}
				X_1^i,X_2^i,\cdots,X_N^i
			\end{bmatrix}.
		\end{aligned}
	\end{equation}
	We then conclude from $u_c^i\in\mathbb{I}_N^\tau(u_d^i)$ that \eqref{AlgorithmTechnical2} holds. Moreover, we follow the \emph{necessity} part of the proof of Theorem~\ref{ThCharacterization} to conclude that \eqref{AlgorithmTechnical1} holds. As a consequence, for all $0\leq i\leq \ell-1$, there exist matrices $U^i$ and $X^i$ such that \eqref{AlgorithmTechnical1} and \eqref{AlgorithmTechnical2} hold, which, together with \eqref{SetUc} and \eqref{SetUci}, implies that $u_c\in\mathcal{U}_c(x_0,u_d)$.
\end{proof}

According to Theorem~\ref{ThInputSetCharacterization}, for given $\ell \in \mathbb{Z}_{>0}$, $x_0\in\mathbb{R}^n$, and $u_d:\llbracket 0,\ell \rrbracket \rightarrow \mathbb{R}^m$, $\mathcal{U}_c(x_0,u_d)$ is the set of all continuous-time inputs $u_c\in\mathbb{I}_N^\tau(u_d)$ that establish \eqref{InterpolatingCondition}. This is crucial as it practically gives a step-by-step procedure for the construction of such $u_c$. In fact, given any $u_d:\llbracket 0,\ell \rrbracket \rightarrow \mathbb{R}^m$, one may follow the first part of the proof of Theorem~\ref{ThInputSetCharacterization} to construct a $u_c\in\mathbb{I}_N^\tau(u_d)$ that establishes \eqref{InterpolatingCondition}. This becomes essential in control synthesis for a continuous-time system on the basis of its discrete-time model (see Section~\ref{Sec_ControlSynthesis}). 
\section{Discretization by System Interpolation}\label{Sec_Discretization}
We now exploit system interpolation to conduct discretization. More specifically, given a continuous-time system \eqref{ContinuousSystem}, for a sampling time $\tau\in\mathbb{R}_{>0}$ and an integer $N\geq 1$, we construct a discrete-time system \eqref{DiscreteSystem} in such a way that \eqref{ContinuousSystem} is an $N$-th order interpolator of \eqref{DiscreteSystem} with respect to the sampling time $\tau$. For this purpose, we make use of Theorem~\ref{ThCharacterization} and seek matrices $A_d$ and $B_d$ such that \eqref{Finalcharacterization} holds. 

We therefore consider the subspace inclusion \eqref{Finalcharacterization}, where matrices $A_d$ and $B_d$ are now regarded as \emph{unknowns}. It follows from \eqref{Finalcharacterization} that there exists a matrix $M \in \mathbb{R}^{(nN+mN)\times (n+2m)}$ such that
\begin{equation*}\label{IntermediateTechnical}
	\begingroup 
	\setlength\arraycolsep{2pt}\begin{bmatrix}
		\psi_0^\top \otimes I & 0\\
		\Psi^\top\oplus(-A_c)&-I \otimes B_c\\
		\Phi^\top \otimes I&0\\
		0 &\Phi^\top \otimes I
	\end{bmatrix} 
	\endgroup M = \begingroup\setlength\arraycolsep{1.25pt}\begin{bmatrix}
		A_c-\dot{\phi}_0(0)I&B_c&0\\
		-\sigma^\top\otimes I&0&0\\
		A_d - \phi_0(\tau)I&B_d&0\\
		0&\!\!\!-\phi_0(\tau)I&I
	\end{bmatrix}\endgroup,
\end{equation*}
which can be written as 
\begin{equation}\label{CompactRep}
	Q_N^\tau (A_c,B_c) M = R_N^\tau(A_c,B_c) + T_1\begin{bmatrix}
		A_d&B_d
	\end{bmatrix}T_2, 
\end{equation}
where
\begin{subequations}\label{QRSV}
	\begin{equation}
		\begin{aligned}
			Q_N^\tau (A_c,B_c) &= \begin{bmatrix}
				\psi_0^\top \otimes I & 0\\
				\Psi^\top\oplus(-A_c)&-I \otimes B_c\\
				\Phi^\top \otimes I&0\\
				0 &\Phi^\top \otimes I
			\end{bmatrix},\\
			R_N^\tau(A_c,B_c) &= \begingroup\setlength\arraycolsep{1.25pt}\begin{bmatrix}
				A_c-\dot{\phi}_0(0)I&B_c&0\\
				-\sigma^\top\otimes I&0&0\\
				- \phi_0(\tau)I&0&0\\
				0&\!\!\!-\phi_0(\tau)I&I
			\end{bmatrix}\endgroup,
		\end{aligned}
	\end{equation}
	and
	\begin{equation}
		\begin{aligned}
			T_1 &= \begin{bmatrix}
				0\\0\\I\\0
			\end{bmatrix}, &T_2&= \begin{bmatrix}
				I&0&0\\
				0&I&0
			\end{bmatrix}. 
		\end{aligned}
	\end{equation}
\end{subequations}
We emphasize that matrices $Q_N^\tau (A_c,B_c)$ and $R_N^\tau(A_c,B_c)$ solely depend on the quadrature nodes (which are determined by the sampling time $\tau$ and the polynomial degree $N$) and the \emph{known} matrices $A_c$ and $B_c$. We also highlight that the representation \eqref{CompactRep} is \emph{linear} in terms of the unknown matrices $M$, $A_d$, and $B_d$. This representation thus enables the following result, which characterizes the existence of matrices $A_d$ and $B_d$ as the solvability of a linear equation that is completely in terms of the quadrature nodes and the system parameters $A_c$ and $B_c$. 
\begin{theorem}\label{Th_Design}
	Given a continuous-time system \eqref{ContinuousSystem}, for a sampling time $\tau\in\mathbb{R}_{>0}$ and an integer $N\geq 1$, there exist matrices $A_d$ and $B_d$ such that \eqref{ContinuousSystem} is an $N$-th order interpolator of \eqref{DiscreteSystem} with respect to the sampling time $\tau$ if and only if there exist vectors $v_1$ and $v_2$ such that 
	\begin{equation}\label{DesignResult}
		\begin{bmatrix}
			I\otimes Q_N^\tau (A_c,B_c) &-T_2^\top \otimes T_1
		\end{bmatrix}\begin{bmatrix}
			v_1\\
			v_2\\
		\end{bmatrix} = \operatorname{vec}\big( R_N^\tau(A_c,B_c) \big).
	\end{equation}
\end{theorem}
\begin{proof}
	To show necessity, we suppose that there exist $A_d$ and $B_d$ such that \eqref{ContinuousSystem} is an $N$-th order interpolator of \eqref{DiscreteSystem} with respect to the sampling time $\tau$. It then follows from Theorem~\ref{ThCharacterization} that \eqref{Finalcharacterization} holds. This, as discussed above, implies the existence of a matrix $M$ such that \eqref{CompactRep} holds. Then, vectorization of \eqref{CompactRep} leads to \eqref{DesignResult} with 
	\begin{equation}\label{IntermediateResult}
		\begin{aligned}
			v_1 &= \operatorname{vec}(M), &v_2&=\operatorname{vec}\Big(\begin{bmatrix}
				A_d&B_d
			\end{bmatrix}\Big).
		\end{aligned}
	\end{equation} 
	
	To prove sufficiency, we suppose that there exist matrices $v_1$ and $v_2$ such that \eqref{DesignResult} holds. We then let matrices $M$, $A_d$, and $B_d$ be such that \eqref{IntermediateResult} hold. We then exploit matrix vectorization to conclude \eqref{CompactRep}. This, together with \eqref{QRSV}, implies \eqref{Finalcharacterization}. It then follows from  Theorem~\ref{ThCharacterization} that \eqref{ContinuousSystem} is an $N$-th order interpolator of \eqref{DiscreteSystem} with respect to the sampling time $\tau$.
\end{proof}

Theorem~\ref{Th_Design} characterizes existence of the discrete-system \eqref{DiscreteSystem} as solvability of the linear equation \eqref{DesignResult}, whose solution gives the parameters of \eqref{DiscreteSystem}, \textit{i.e.,} matrices $A_d$ and $B_d$.

In the next section, we make use of the discrete-time system \eqref{DiscreteSystem}, obtained according to Theorem~\ref{Th_Design}, to conduct control synthesis for the continuous-time system \eqref{ContinuousSystem}.
\section{System Interpolation for Control Synthesis}\label{Sec_ControlSynthesis}
As discussed earlier, hierarchical control schemes that split synthesis into discretization, planning, and execution do not formally guarantee that the continuous-time system fully complies with a specification, as they only guarantee that its state trajectories remain sufficiently close to an interpolation of the planned trajectory (which satisfies the specification only at particular time instances rather than all times). In fact, such schemes often do not provide any measure of the extent to which the specification is satisfied/violated in between these instances. Motivated by this, we utilize system interpolation in control synthesis to 1) ensure that the controlled continuous-time trajectories adhere to the specification at \emph{each} sampled time and 2) measure to what extent they may violate the specification in between the sampled times. 

We therefore consider the continuous-time system \eqref{ContinuousSystem} and use Theorem~\ref{Th_Design} to construct its discrete-time model \eqref{DiscreteSystem} for some sampling time $\tau\in\mathbb{R}_{>0}$ and integer $N\geq 1$. We thus emphasize that \eqref{ContinuousSystem} is an $N$-th order interpolator of \eqref{DiscreteSystem} with respect to $\tau$. 

Given a specification $\mathscr{S}$, for an initial condition $x_0\in\mathbb{R}^n$, suppose that (for some $\ell\in\mathbb{Z}_{>0}$) the control sequence $u_d:\llbracket0,\ell\rrbracket \rightarrow\mathbb{R}^m$ is already designed in such a way that $x_d(\cdot;x_0,u_d)$ adheres to the specification $\mathscr{S}$. This induces a path of regions
\begin{equation}\label{Path}
	\Pi^\ell := \pi_0\pi_1\cdots\pi_\ell, 
\end{equation}
where the \emph{regions of interest} $\pi_0,\pi_1,\cdots,\pi_\ell\subset\mathbb{R}^n$, which are dictated by the specification $\mathscr{S}$, are such that 
\begin{equation}\label{DiscreteRegionSatisfaction}
	\forall i\in\llbracket0,\ell\rrbracket: \quad x_d(i;x_0,u_d) \in \pi_i.
\end{equation}
The path $\Pi^\ell$ therefore \emph{accepts} the specification $\mathscr{S}$ in the sense that the regions of interest $\pi_0,\pi_1,\cdots,\pi_\ell$ and the order according to which they appear conform to the specification $\mathscr{S}$ (for further details, see, \textit{e.g.,} \cite{4459804,belta2017formal,8657716}).
\begin{remark}
	We emphasize that we are not concerned with the design of the control sequence $u_d$, as it depends on the specification $\mathscr{S}$. In accordance to the expression of $\mathscr{S}$, the control sequence $u_d$ can be designed by techniques such as dynamic programming \cite{bertsekas2012dynamic}, MPC \cite{kouvaritakis2016model}, mixed integer linear/quadratic programming (MILP/MIQP) \cite{belta2019formal,7039363, 6907641}, and symbolic control \cite{YIN2024100940,belta2017formal}. \demo
\end{remark}

We make use of the initial condition $x_0$ and the discrete-time input $u_d$ obtained above to construct the set of continuous-time inputs $\mathcal{U}_c(x_0,u_d)$ as in \eqref{SetUc}. Since $\bm{\Sigma}_c$ is an $N$-th order interpolator of $\bm{\Sigma}_d$ with respect to $\tau$. It follows from Theorem~\ref{ThInputSetCharacterization} that any $u_c\in\mathcal{U}_c(x_0,u_d)$ establishes \eqref{InterpolatingCondition}. This, as a consequence of \eqref{DiscreteRegionSatisfaction}, implies that for all $u_c\in\mathcal{U}_c(x_0,u_d)$, we have that 
\begin{equation}\label{ContinuousRegionSatisfaction}
	\forall i\in \llbracket0,\ell\rrbracket: \quad x_c(i\tau;x_0,u_c) \in \pi_i,
\end{equation}
\textit{i.e.,} the controlled trajectory $x_c(\cdot;x_0,u_c)$ satisfies the specification $\mathscr{S}$ at \emph{each} sampled time $t = i\tau$ (for $i=0,1,\cdots,\ell$). This therefore indicates that by choosing the continuous-time control input $u_c\in\mathcal{U}_c(x_0,u_d)$, one ensures that the controlled continuous-time trajectories adhere to the specification at each sampled time.  

We now employ \eqref{ContinuousRegionSatisfaction} to measure the (possible) violation of $\mathscr{S}$ in between sampled times. Towards that end, we first define the \emph{distance} of a point $z^*\in\mathbb{R}^n$ from a region $\pi\subset\mathbb{R}^n$, denoted by $\operatorname{d}(z^*,\pi)$, as 
\begin{equation}\label{Distance}
	\operatorname{d}(z^*,\pi) = \inf_{z\in\pi}|z^*-z|.
\end{equation} 
We accordingly define the \emph{Hausdorff} distance \cite{magnus2021metric} between two regions $\pi,\pi'\subset\mathbb{R}^n$, denoted by $\operatorname{dH}(\pi,\pi')$, as 
\begin{equation}\label{HausdorffDistance}
	\operatorname{dH}(\pi,\pi') = \max \left\{\sup_{z\in\pi} \operatorname{d}(z,\pi'), \sup_{z'\in\pi'} \operatorname{d}(z',\pi) \right\}.
\end{equation}
Then, for a point $z^*\in\mathbb{R}^n$ and regions $\pi,\pi'\subset\mathbb{R}^n$, we observe that 
\begin{equation}\label{DistanceBound}
	\operatorname{d}(z^*,\pi) \leq \operatorname{d}(z^*,\pi')+\operatorname{dH}(\pi,\pi'),
\end{equation}
which indicates that the distance $\operatorname{d}(z^*,\pi)$ can be estimated in terms of the distance $\operatorname{d}(z^*,\pi')$ and the Hausdorff distance $\operatorname{dH}(\pi,\pi')$ (for further details, see, \textit{e.g.,} \cite[Section 1.2.4]{magnus2021metric}). 

Between two consecutive sampled times, we will use the estimation \eqref{DistanceBound} to measure the (possible) violation of $\mathscr{S}$. Taking any $u_c\in\mathcal{U}_c(x_0,u_d)$, we note that \eqref{AlgorithmTechnical3} and \eqref{Technical6} hold for all $i\in\llbracket0,\ell-1\rrbracket$ and all $t\in[0,\tau]$. Thus, for any $j\in\{i,i+1\}$, it follows from \eqref{ContinuousRegionSatisfaction} that
\begin{equation*}
	\begin{aligned}
		\operatorname{d}\!\big(x_c(i\tau+t;x_0,u_c),\pi_j\big)\!&=\!\inf_{z\in\pi_j} \vert x_c(i\tau+t;x_0,u_c) -z\vert \\
		&\leq\!\big\vert x_c(i\tau+t;x_0,u_c)\!-\! x_c(j\tau;x_0,u_c)\big\vert \\
		&=\!\left\vert \int_{j\tau}^{i\tau+t} \dot{x}_c(\hat{t};x_0,u_c) \dhatt \right\vert\\
		&\leq\!\left\vert \int_{j\tau}^{i\tau+t} \vert \dot{x}_c(\hat{t};x_0,u_c) \vert \dhatt\right\vert\\
		&\leq\!\tau\norm{\dot{x}_c(\cdot;x_0,u_c)}_{\mathcal{L}_\infty^n[i\tau,(i+1)\tau]},
	\end{aligned}
\end{equation*}
which, as a consequence of \eqref{ContinuousSystem}, results in
	\begin{align}
		\operatorname{d}\!\big(x_c(i\tau+t;&x_0,u_c),\pi_j\big) \label{Technical4}\\
		&\leq\tau\norm{A_cx_c(\cdot;x_0,u_c) +Bu_c}_{\mathcal{L}_\infty^n[i\tau,(i+1)\tau]}.\nonumber
	\end{align}
Considering that 
\begin{equation*}
	\restr{\big(A_cx_c(\cdot;x_0,u_c) +Bu_c\big)}{[i\tau,(i+1)\tau]}\in\restr{\mathbb{P}_N^n}{[i\tau,(i+1)\tau]},
\end{equation*}
we now recall from \cite[Theorem 4.9.6]{timan2014theory} that for any polynomial $P\in\restr{\mathbb{P}_N^n}{[i\tau,(i+1)\tau]}$, we have that
\begin{equation*}
	\norm{P}_{\mathcal{L}_\infty^n[i\tau,(i+1)\tau]} \leq N\left(\frac{6}{\tau}\right)^{\frac{1}{2}}\norm{P}_{\mathcal{L}_\infty^2[i\tau,(i+1)\tau]},
\end{equation*}
which, together with \eqref{Technical4}, implies that
\begin{align*}
		\operatorname{d}\!\big(&x_c(i\tau+t ;x_0,u_c),\pi_j\big)\\
		&\leq N(6\tau)^{\frac{1}{2}}  \norm{A_cx_c(\cdot;x_0,u_c) +Bu_c}_{\mathcal{L}_2^n[i\tau,(i+1)\tau]}\nonumber\\
		&= N(6\tau)^{\frac{1}{2}} \left(\int_{i\tau}^{(i+1)\tau}\left\vert A_cx_c(\hat{t};x_0,u_c) +Bu_c(\hat{t})\right\vert^2\dhatt\right)^{\frac{1}{2}}\nonumber.
\end{align*}
 This, together with \eqref{PropTechnical4}, gives
 \begin{align}
 	\operatorname{d}\!\big(&x_c(i\tau+t ;x_0,u_c),\pi_j\big) \label{MidTechnical0}\\
 	&\leq N(6\tau)^{\frac{1}{2}}\left(\int_{0}^{\tau}\left\vert A_cx_c(i\tau+\hat{t};x_0,u_c) + B_cu_c^i(\hat{t}) \right\vert^2\dhatt\right)^{\frac{1}{2}}.\nonumber
 \end{align}
 We then note from \eqref{AlgorithmTechnical3} that
 \begin{equation*}
 	\restr{x_c(i\tau+\hat{t};x_0,u_c)}{[0,\tau]} \in \restr{\mathbb{P}_N^n}{[0,\tau]}.
 \end{equation*}
  Furthermore, it follows from \eqref{Technical6} that 
 \begin{equation*}
 	x_c(i\tau;x_0,u_c) = x_d(i;x_0,u_d),
 \end{equation*}
 and that 
 \begin{equation*}
 	x_c(i\tau+t_{N,j};x_0,u_c) = X_j^i,
 \end{equation*}
 for all $j=1,2,\ldots,N$. This therefore enables the utilization of \eqref{NormComputation} to write \eqref{MidTechnical0} as
 \begin{align}
 	\operatorname{d}\!\big(&x_c(i\tau+t ;x_0,u_c),\pi_j\big)\\
 	&\leq N(6\tau)^{\frac{1}{2}} \bigg( \big\vert A_cx_d(i;x_0,u_d)+B_cu_d(i)\big\vert^2w_{N,0}\nonumber\\
 	& \hspace{2.5cm} + \sum_{j=1}^{N}\big\vert A_cX_j^i + B_cU_j^i \big \vert^2w_{N,j}\bigg)^{\frac{1}{2}},\nonumber
 \end{align}
which, after the substitution of \eqref{QuadratureWeight}, gives
\begin{align}
		\operatorname{d}\!\big(&x_c(i\tau+t ;x_0,u_c),\pi_j\big) \label{MidInequality1}\\ 
		&\hspace{1cm}\leq\frac{N(6\tau)^{\frac{1}{2}}}{N+1}\Bigg(\big\vert A_cx_d(i;x_0,u_d)+B_cu_d(i)\big\vert ^2 \tau\nonumber\\
		&\hspace{1.5cm} + \sum_{j=1}^{N}\left\vert A_cX_j^i+B_cU_j^i\right\vert^2\frac{\tau-t_{N,i}}{|\mathfrak{L}_N^\tau(t_{N,i})|^2} \Bigg)^{\frac{1}{2}}.\nonumber
\end{align}
To obtain a compact representation of the bound, we define the matrix 
\begin{equation*}
	W_N^\tau = \begingroup 
	\setlength\arraycolsep{2pt}\begin{bmatrix}
		\frac{\tau-t_{N,1}}{|\mathfrak{L}_N^\tau(t_{N,1})|^2}I&0&\cdots&0\\
		0&\frac{\tau-t_{N,2}}{|\mathfrak{L}_N^\tau(t_{N,2})|^2}I&\cdots&0\\
		\vdots&\vdots&\ddots&\vdots\\
		0&0&\cdots&\frac{\tau-t_{N,N}}{|\mathfrak{L}_N^\tau(t_{N,N})|^2}I
	\end{bmatrix}\endgroup,
\end{equation*}
and accordingly construct the matrix 
\begin{equation}\label{Delta}
	\Delta_N^\tau(A_c,B_c)\!=\!\begingroup 
	\setlength\arraycolsep{1pt}\begin{bmatrix}
		I\otimes A_c&I\otimes B_c
	\end{bmatrix}\endgroup^\top\!W_{N}^\tau \begingroup\setlength\arraycolsep{1pt}\begin{bmatrix}
		I\otimes A_c&I\otimes B_c
	\end{bmatrix}\endgroup\!. 
\end{equation}
It now follows from \eqref{MidInequality1} that 
\begin{equation}\label{MidInequality2}
	\begin{aligned}
		\operatorname{d}\!\big(x_c(i\tau+t;x_0,&u_c),\pi_j\big)\\
		& \leq \frac{6^{\frac{1}{2}}\tau N}{N+1}\big\vert A_cx_d(i;x_0,u_d)+B_cu_d(i)\big\vert \\
		&\quad +\frac{(6\tau)^{\frac{1}{2}}N}{N+1} \left \vert \begin{bmatrix}
			\operatorname{vec}(X^i)\\\operatorname{vec}(U^i)
		\end{bmatrix}\right\vert_{\Delta_N^\tau (A_c,B_c)}.
	\end{aligned}
\end{equation}
We utilize \eqref{MidInequality2} to obtain the following result, which measures the distance of a state value in between two sampled times from a region $\pi\subset\mathbb{R}^n$. 
\begin{theorem}\label{ThMeasurement}
	Let $\bm{\Sigma}_c$ be an $N$-th order interpolator of $\bm{\Sigma}_d$ with respect to a sampling time $\tau$. Given regions of interest $\pi_0,\pi_1,\cdots,\pi_\ell\subset\mathbb{R}^n$ and an initial condition $x_0\in\mathbb{R}^n$, for some $\ell \in \mathbb{Z}_{>0}$, let the control sequence $u_d:\llbracket0,\ell\rrbracket \rightarrow\mathbb{R}^m$ be such that \eqref{DiscreteRegionSatisfaction} holds. Take any continuous-time control input $u_c\in\mathcal{U}_c(x_0,u_d)$. For any integer $0\leq i\leq \ell-1$, let the function $u_c^i \in \mathcal{U}_c^i(x_0,u_d)$ and the matrices $U^i\in\mathbb{R}^{m\times N}$ and $X^i\in\mathbb{R}^{n\times N}$ be such that \eqref{PropTechnical4}, \eqref{AlgorithmTechnical1}, and \eqref{AlgorithmTechnical2} hold. Then, given a region $\pi\subset\mathbb{R}^n$, for all $t\in[0,\tau]$, we have that 
		\begin{align}
			\operatorname{d}\!\big(x_c(i\tau+t;x_0,&u_c),\pi\big)\nonumber\\
			& \leq \frac{6^{\frac{1}{2}}\tau N}{N+1}\big\vert A_cx_d(i;x_0,u_d)+B_cu_d(i)\big\vert \nonumber\\
			&\quad + \frac{(6\tau)^{\frac{1}{2}}N}{N+1} \left \vert \begin{bmatrix}
				\operatorname{vec}(X^i)\\\operatorname{vec}(U^i)
			\end{bmatrix}\right\vert_{\Delta_N^\tau (A_c,B_c)}\label{DistCondition} \\
			&\quad + \min \left\{\operatorname{dH}\big(\pi_i,\pi \big), \operatorname{dH}\big(\pi_{i+1},\pi \big)\right\},\nonumber
		\end{align}
	where $\Delta_N^\tau(A_c,B_c)$ is defined as in \eqref{Delta}.
\end{theorem}
\begin{proof}
	Let $i\in\llbracket0,\ell\rrbracket$ and $t\in[0,\tau]$. Then, for $j\in\{i,i+1\}$, it follows from \eqref{DistanceBound} that
	\begin{equation*}
		\begin{aligned}
			\operatorname{d}\!\big(x_c(i\tau+t;x_0,u_c),\pi\big) \leq & \, \operatorname{d}\!\big(x_c(i\tau+t;x_0,u_c),\pi_j\big)\\
			&+ \operatorname{dH}\!\big(\pi_j,\pi\big).
		\end{aligned}
	\end{equation*}
	which, as a result of \eqref{MidInequality2}, immediately gives \eqref{DistCondition}.
\end{proof}
\begin{remark}		
Theorem~\ref{ThMeasurement} gives an \emph{upper bound}, on the distance of state values (between two consecutive sampled times) from a given region, that depends on the solutions of \eqref{AlgorithmTechnical1}. More specifically, given an initial condition $x_0\in\mathbb{R}^n$ and a control sequence $u_d:\llbracket0,\ell\rrbracket \rightarrow\mathbb{R}^m$ that establishes \eqref{DiscreteRegionSatisfaction} (for some regions of interest $\pi_0,\pi_1,\cdots,\pi_\ell\subset\mathbb{R}^n$), for any region $\pi\subset\mathbb{R}^n$, it follows from Theorem~\ref{ThMeasurement} that any continuous-time control input $u_c\in\mathcal{U}_c(x_0,u_d)$ enforces \eqref{DistCondition}, which depends on matrices $X^i$ and $U^i$ that satisfy \eqref{AlgorithmTechnical1}. Considering that these matrices are \emph{not} necessarily unique, \eqref{DistCondition} suggests that one may \emph{reduce} the distance between $x_c(i\tau+t;x_0,u_c)$ and $\pi$ by constructing $u_c$ on the basis of matrices $X^i$ and $U_i$ that are obtained by minimizing 
\begin{equation*}
	\begin{aligned}
		\left\vert \begin{bmatrix}
			\operatorname{vec}(X^i)\\\operatorname{vec}(U^i)
		\end{bmatrix}\right\vert_{\Delta_N^\tau (A_c,B_c)}
	\end{aligned}
\end{equation*}
with respect to \eqref{AlgorithmTechnical1}. We note that such minimization problem can be cast into quadratic optimization with equality constraints, which admits a closed-form solution. Thus, by solving this optimization problem for all $0\leq i\leq \ell-1$ and constructing $u_c$ accordingly, one enforces the state $x_c(i\tau+t;x_0,u_c)$ to remain as \emph{close} as possible to the region $\pi$. \demo
\end{remark}

Measuring the distance of a state value in between sampled times from a given region in the state-space, Theorem~\ref{ThMeasurement} gives a measure of the extent to which a controlled continuous-time trajectory may violate a given specification. We further demonstrate this in the following examples. 
\begin{example}
	We consider a safety-critical control problem. Given a `safe' region $\pi \subset\mathbb{R}^n$ and an initial condition $x_0\in\pi$, for some $\ell\in\mathbb{Z}_{>0}$, let the control sequence $u_d: \llbracket0,\ell\rrbracket\rightarrow\mathbb{R}^m$ be such that 
	\begin{equation}\label{Safety}
		\forall i\in \llbracket0,\ell\rrbracket:\quad x_d(i;x_0,u_d) \in \pi,
	\end{equation} 
	\textit{i.e.,} subject to the control sequence $u_d$, the discrete-time model $\bm{\Sigma}_d$ admits a state trajectory that remains within the safe region over the time interval $\llbracket0,\ell\rrbracket$. 
	We note that such $u_d$ can be obtained using techniques such as \emph{dynamic programming} \cite{bertsekas2012dynamic} and \emph{set-invariance} control \cite{blanchini2008set}.
	
	We observe that \eqref{Safety} is obtained from \eqref{DiscreteRegionSatisfaction} by choosing $\pi_i = \pi$ for all $i=0,1,\ldots,\ell$. Thus, by taking any $u_c\in\mathcal{U}_c(x_0,u_d)$, we conclude from Theorem~\ref{ThMeasurement} that for any integer $0\leq i\leq \ell-1$, 
	\begin{align}
			\forall t\in[0,\tau]: \quad \operatorname{d}\big(&x_c(i\tau+t;x_0,u_c),\pi\big) \label{SmallBound} \\
			&\leq \frac{6^{\frac{1}{2}}\tau N}{N+1}\big\vert A_cx_d(i;x_0,u_d)+B_cu_d(i)\big\vert \nonumber \\
			&\quad + \frac{(6\tau)^{\frac{1}{2}}N}{N+1} \left \vert \begin{bmatrix}
				\operatorname{vec}(X^i)\\\operatorname{vec}(U^i)
			\end{bmatrix}\right\vert_{\Delta_N^\tau (A_c,B_c)},\nonumber
	\end{align}
	which gives an upper bound on the distance between the state $x_c(i\tau+t;x_0,u_c)$ and the safe region $\pi$. Importantly, \eqref{SmallBound} indicates that the distance from $x_c(i\tau+t;x_0,u_c)$ to $\pi$ decreases with respect to the sampling time $\tau$. This basically means that for a sufficiently small $\tau$, the state $x_c(i\tau+t;x_0,u_c)$ remains desirably close to the safe region $\pi$. As \eqref{SmallBound} holds for all $0\leq i\leq \ell-1$, we conclude that for a sufficiently small $\tau$, the controlled continuous-time trajectory remains desirably close to the safe region $\pi$ over the time interval $[0,\ell\tau]$.
\end{example}
\begin{example}
	We now consider a signal temporal logic (STL) control problem. For a vector $\Gamma \in \mathbb{R}^n$ and a scalar $\gamma\in\mathbb{R}$, we define the predicate $\mu$ as
	\begin{equation*}
		\mu(x):=\begin{cases}
			\mathrm{True}, \quad \, \text{if} \;\; \Gamma^\top x + \gamma\geq 0,\\
			\mathrm{False}, \quad \text{if} \;\; \Gamma^\top x + \gamma< 0,
		\end{cases}
	\end{equation*}
	and consider the specification 
	\begin{equation*}
		\mathscr{S} = \mathrm{G}_{[t_1,t_2]} \mu,
	\end{equation*}
	where $0\leq t_1 < t_2$ (for further details on STL, see, \textit{e.g.,} \cite{lindemann2025formal}). We note that the formula $\mathscr{S}$ requires the predicate $\mu$ to \emph{always} remain true over the time interval $[t_1,t_2]$. Given an initial condition $x_0\in\mathbb{R}^n$, for some $\ell\in\mathbb{Z}_{>0}$, let the control sequence $u_d: \llbracket0,\ell\rrbracket\rightarrow\mathbb{R}^m$ be such that
	\begin{equation*}
		x_d(0; x_0,u_d) \models \mathrm{G}_{\llbracket i_1,i_2\rrbracket} \mu,
	\end{equation*}
	where $i_1,i_2\in\llbracket0,\ell\rrbracket$ are such that $i_1\tau\leq t_1<t_2 \leq i_2\tau$. This implies that 
	\begin{equation}\label{ExmpTechnical1}
		\forall i\in\llbracket i_1,i_2\rrbracket: \quad \mu\big(x_d(i; x_0,u_d)\big) = \mathrm{True},
	\end{equation}
	which induces a path \eqref{Path} such that 
	\begin{equation*}
		\forall i\in\llbracket i_1,i_2\rrbracket:\quad \pi_i = \big\{x_d(i; x_0,u_d)\big\},
	\end{equation*}
	and that \eqref{DiscreteRegionSatisfaction} holds. Then, taking any $u_c\in\mathcal{U}_c(x_0,u_d)$, for any $i_1\leq i\leq i_2$, we let 
	\begin{equation*}
		\pi = \big\{x_d(i; x_0,u_d)\big\},
	\end{equation*}
	and immediately conclude from Theorem~\ref{ThMeasurement} that \eqref{SmallBound} holds. 
	
	We now define the STL score (see, \textit{e.g.,} \cite{belta2019formal,lindemann2025formal}) of $x_c(t;x_0,u_c)$ with respect to the predicate $\mu$ as 
	\begin{equation*}
		\rho\big(x_c(t;x_0,u_c)\big) = \Gamma^\top x_c(t;x_0,u_c) + \gamma.
	\end{equation*}
	We let $i\in\llbracket0,\ell\rrbracket$ and $\hat{t}\in[0,\tau]$ be such that $t = i\tau+\hat{t}$. Recalling that $u_c\in\mathcal{U}_c(x_0,u_d)$, we use \eqref{CompleteSet} to conclude \eqref{InterpolatingCondition}, which implies that $x_c(i\tau;x_0,u_c) = x_d(i;x_0,u_d)$. 
	This then enables us to write 
	\begin{equation*}
		\begin{aligned}
			\rho\big(x_c(t;x_0,u_c)\big) =& \  \Gamma^\top \big(x_c(i\tau+\hat{t};x_0,u_c) - x_d(i;x_0,u_d)\big)\\
			&+ \rho\big(x_d(i\tau;x_0,u_d)\big),
		\end{aligned}
	\end{equation*}
	which, together with \eqref{SmallBound} and \eqref{ExmpTechnical1}, gives
	\begin{align}
			\Big\vert \rho\big(x_c(t;x_0,u_c)\big) &- \rho\big(x_d(i\tau;x_0,u_d)\big) \Big \vert \nonumber\\
			&\leq  \vert \Gamma \vert \operatorname{d}\big(x_c(i\tau+\hat{t};x_0,u_c),\pi\big)\label{ExmpTechnical2}\\
			&\leq \frac{6^{\frac{1}{2}}\tau N \vert \Gamma \vert}{N+1}\big\vert A_cx_d(i;x_0,u_d)+B_cu_d(i)\big\vert \nonumber\\
			&\quad + \frac{(6\tau)^{\frac{1}{2}}N \vert \Gamma \vert}{N+1} \left \vert \begin{bmatrix}
				\operatorname{vec}(X^i)\\\operatorname{vec}(U^i)
			\end{bmatrix}\right\vert_{\Delta_N^\tau (A_c,B_c)}.\nonumber
	\end{align}
	We observe that for any $i_1\leq i\leq i_2$, \eqref{ExmpTechnical2} gives an upper bound on the deviation from the discrete-time STL score with respect to the predicate $\mu$. We therefore emphasize that \eqref{ExmpTechnical2} gives a measure of the extent to which the continuous-time trajectory $x_c(\cdot;x_0,u_d)$ may violate $\mathscr{S}$ (between two consecutive sampled times). We further note from \eqref{ExmpTechnical2} that for a sufficiently small sampling time $\tau$, the continuous-time STL score $\rho\big(x_c(t;x_0,u_c)\big)$ becomes desirably close to its discrete-time counterpart $\rho\big(x_d(i\tau;x_0,u_d)\big)$. This indicates that by choosing the sampling time $\tau$ sufficiently small, the extent of (possible) specification violation becomes desirably small. 
\end{example}

To recapitulate, employing system interpolation for control synthesis enables us to measure to what extent a specification may be violated in between sampled times. In fact, given an initial condition $x_0$ and a control sequence $u_d$ that enforces a discrete-time model $\bm{\Sigma}_d$ of a continuous-time system $\bm{\Sigma}_c$ to fulfill a specification $\mathscr{S}$, choosing any continuous-time control input $u_c\in\mathcal{U}_c(x_0,u_d)$ enables us to 1) ensure that the controlled state trajectory satisfies $\mathscr{S}$ at each sampled time (as indicated by \eqref{ContinuousRegionSatisfaction}), and 2) measure to what extent this trajectory may violate $\mathscr{S}$ in between sampled time (as shown in Theorem~\ref{ThMeasurement}). 
\section{Numerical Simulation}\label{Sec_Simulation}
We consider temporal logic control of a robotic agent with \emph{double-integrator} dynamics. Subject to a control force $F_c$, the robot $\bm{\Sigma}_c$ experiences the velocity $v$ and the displacement $x$. This is captured by the equation of motion 
\begin{equation*}
	\begin{aligned}
		\dot{x}(t) &= v(t),\\
		\dot{v}(t) &= \frac{1}{m_c} F_c(t),
	\end{aligned}
\end{equation*}
where $m_c$ denotes the mass of the robot. We suppose that the robot is initially at full rest, \textit{i.e.,} $x(0) = 0$ and $v(0) = 0$. The objective is to apply the control force $F_c$ such that the robot $\bm{\Sigma}_c$ evolves according to the STL formula
\begin{equation*}
	\begin{aligned}
		\mathscr{S}_c = & \ \mathrm{F}_{[0.2,0.8]}\big(x(t) \leq -2 \big) \\
		&\wedge \mathrm{F}_{[1,1.4]}\big(x(t) \geq 2 \big)\\
		& \wedge \mathrm{F}_{[1.6,2]}\big(x(t) \leq -2 \big)\\
		& \wedge \mathrm{G}_{[0,2]} \big(v(t) \leq 15 \wedge v(t) \geq -15\big).
	\end{aligned}
\end{equation*}
Specifically, the formula $\mathscr{S}_c$ requires that
\begin{enumerate}
	\item the displacement $x$ \emph{eventually} becomes less that $-2$ between $0.2$ and $0.8$ seconds;
	\item the displacement $x$ \emph{eventually} surpasses $2$ between $1$ and $1.4$ seconds;
	\item the displacement $x$ \emph{eventually} drops below $-2$ between $1.6$ and $2$ seconds;
	\item the velocity $v$ \emph{always} remains between $-15 \unit{\meter/\second}$ and $15 \unit{\meter/\second}$. 
\end{enumerate}
We will utilize system interpolation to determine the control force $F_c$ subject to which the robot $\bm{\Sigma}_c$ fulfills these requirements. 

Taking $x_c = \operatorname{col}(x,v)$ and $u_c = \frac{1}{m_c}F_c$, we immediately observe that $\bm{\Sigma}_c$ is governed by the dynamics \eqref{ContinuousSystem} with 
\begin{equation*}
	\begin{aligned}
		A_c &= \begin{bmatrix}
			0&1\\
			0&0
		\end{bmatrix}, &B_c&=\begin{bmatrix}
		0\\
		1
		\end{bmatrix}.
	\end{aligned}
\end{equation*}
 We will therefore establish the objective by synthesizing the continuous-time control $u_c$ such that 
\begin{equation}\label{Technical8}
	x_c(0; 0, u_c) \models \mathscr{S}_c.
\end{equation}
We accomplish this by splitting the continuous-time control synthesis into the following steps.

First, we discretize the continuous-time dynamics \eqref{ContinuousSystem}. After choosing the sampling time $\tau = 0.2\unit{s}$ and the integer $N = 5$, we conclude from Theorem~\ref{Th_Design} that there exist matrices $A_d$ and $B_d$ such that \eqref{ContinuousSystem} is an $N$-th order interpolator of \eqref{DiscreteSystem}. We then utilize Theorem~\ref{Th_Design} to obtain
\begin{equation*}
	\begin{aligned}
		A_d&= \begin{bmatrix}
			0.6990&0.1398\\
			0.0000&0.6990
		\end{bmatrix}, &B_d&=\begin{bmatrix}
		0.0000\\0.1398
		\end{bmatrix}.
	\end{aligned}
\end{equation*} 

As the next step, we design a control sequence $u_d$ such that \eqref{DiscreteSystem} satisfies the formula $\mathscr{S}_c$ at sampled times. Discretizing $\mathscr{S}_c$ with respect to the sampling time $0.2\unit{\second}$, we obtain the (discrete-time) STL formula 
\begin{equation*}
	\begin{aligned}
		\mathscr{S}_d =& \ \mathrm{F}_{\llbracket 1,4\rrbracket}\big(x_{d,1}(i) \leq -2 \big) \\
		&\wedge \mathrm{F}_{\llbracket 5,7\rrbracket}\big(x_{d,1}(i) \geq 2 \big)\\
		& \wedge \mathrm{F}_{\llbracket 8,10\rrbracket}\big(x_{d,1}(i) \leq -2 \big)\\
		& \wedge \mathrm{G}_{\llbracket 0,10\rrbracket} \big(x_{d,2}(i) \leq 15 \wedge x_{d,2}(i) \geq -15\big),
	\end{aligned}
\end{equation*}
where $x_{d,1}$ and $x_{d,2}$ respectively denote the first and second components of $x_d$, \textit{i.e.,} $x_d = \operatorname{col}(x_{d,1},x_{d,2})$.
We then adopt the approach taken in \cite{7039363} to encode the formula $\mathscr{S}_d$ as mixed-integer linear constraints on the control sequence $u_d$ (for further details, see \cite[Section IV.B]{7039363} and \cite[Section 5.2]{belta2019formal}). This then leads to a mixed-integer linear programming feasibility problem whose solution gives the control sequence $u_d:\llbracket0,10\rrbracket\rightarrow\mathbb{R}$ such that
\begin{equation*}
	x_d(0;x_0,u_d)\models\mathscr{S}_d. 
\end{equation*}
The mixed-integer linear program is solved using the optimization toolbox of  MATLAB R2022b. 
\begin{figure}
	\centering
	\begin{subfigure}[b]{0.5\textwidth}
		\includegraphics[width=1\textwidth]{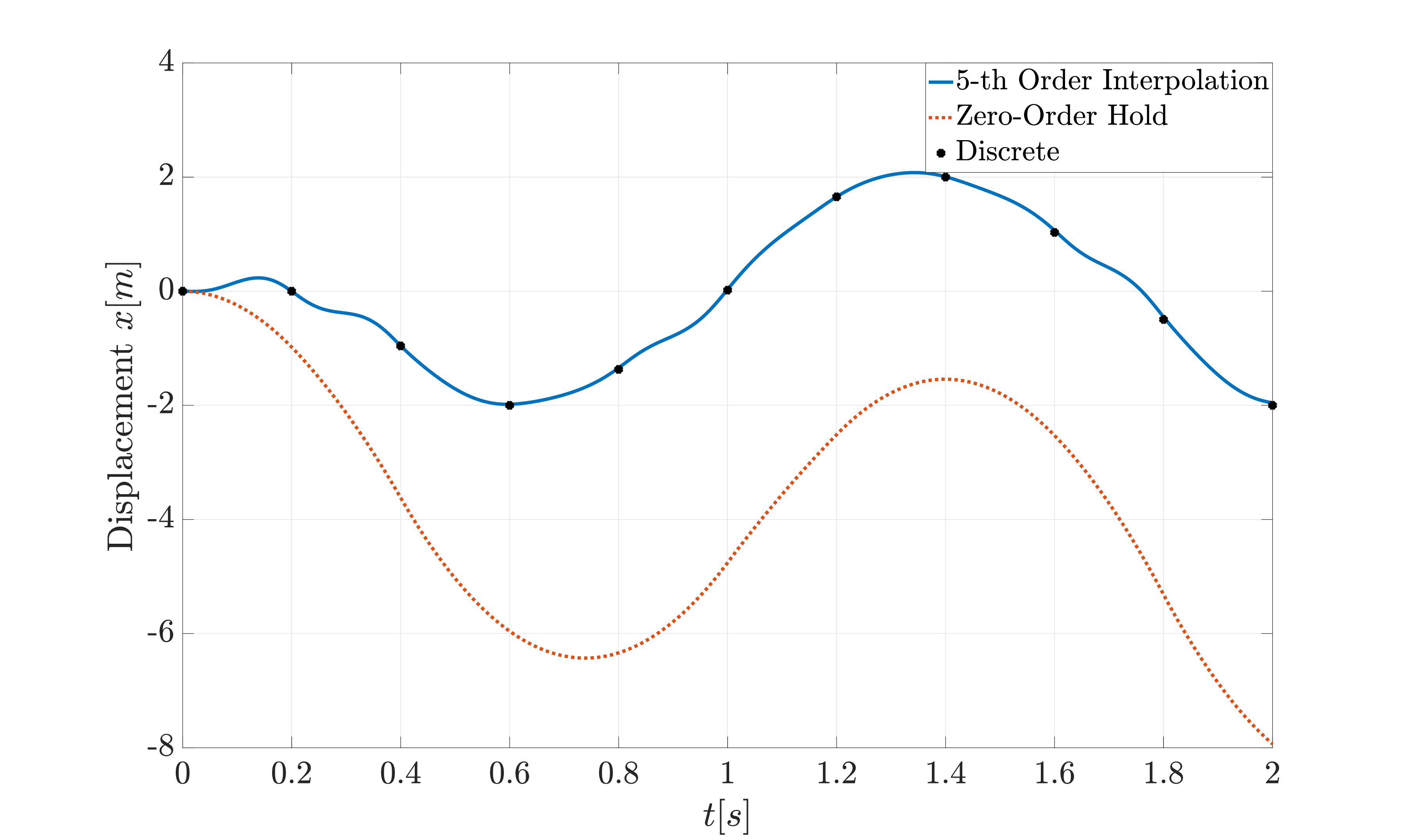}
		\caption{Displacement $x$.}
		\label{Displacement}
	\end{subfigure}
	\vfill
	\begin{subfigure}[b]{0.5\textwidth}
		\includegraphics[width=1\textwidth]{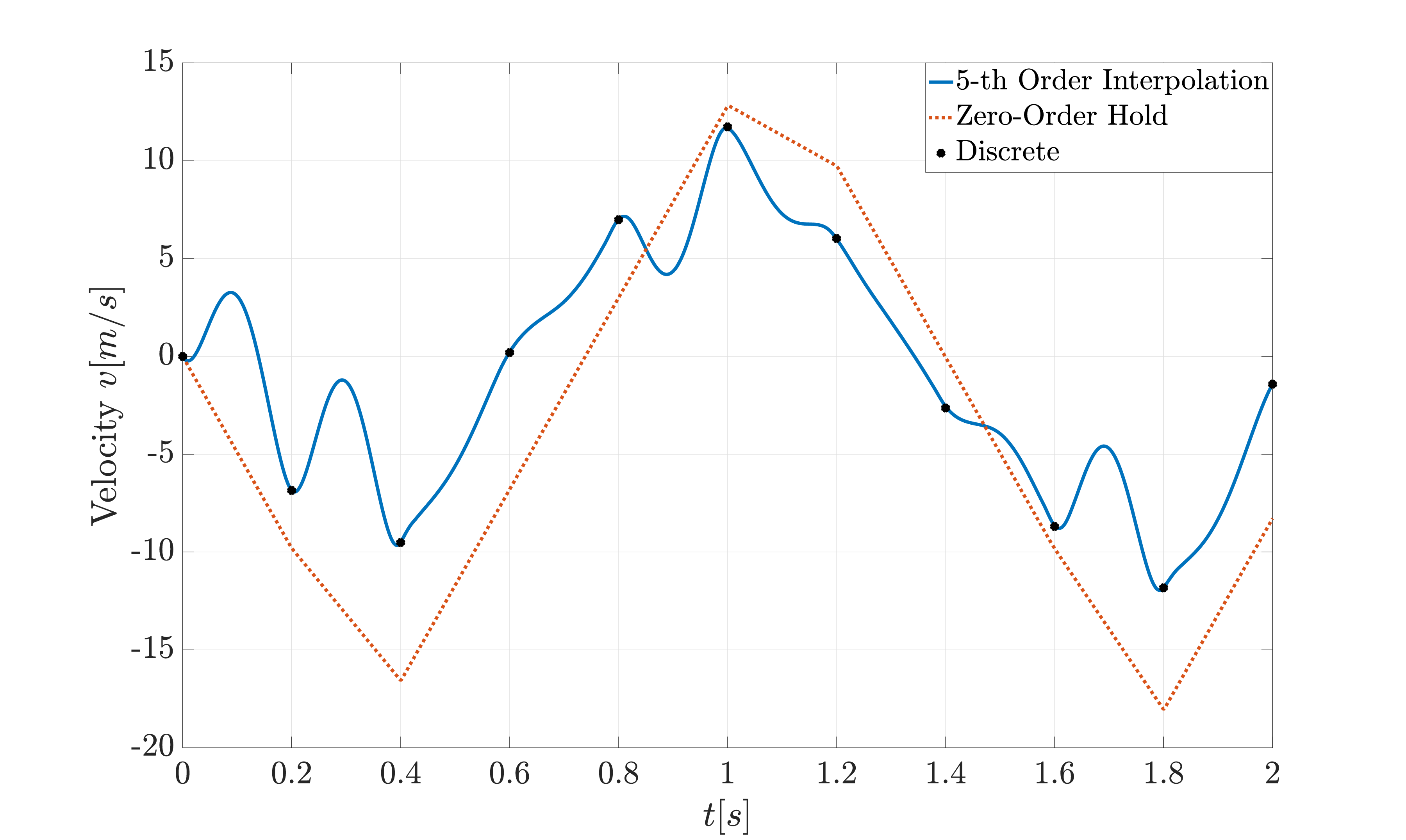}
		\caption{Velocity $v$.}
		\label{Velocity}
	\end{subfigure}
	\caption{Subject to the continuous-time control $u_c\in\mathcal{U}_c(0,u_d)$ (indicated by solid lines in Figure~\ref{Inputs}), the robot $\bm{\Sigma}_c$ experiences a displacement and a velocity that comply with the requirements specified by the STL formula $\mathscr{S}_c$. By contrast, subject to a continuous-time control $u_c$ obtained by ZOH interpolation (indicated by dashed lines in Figure~\ref{Inputs}), neither the displacement of the robot nor its velocity adhere to the formula $\mathscr{S}_c$.}
	\label{Outputs}
\end{figure}

As the final step, we use the designed control sequence $u_d$ to construct the set of continuous-time inputs $\mathcal{U}_c(0,u_d)$, as in \eqref{SetUc}. We note that for $\tau = 0.2$ and $N = 5$, for any $i\in\llbracket0,9\rrbracket$, the matrix equation \eqref{AlgorithmTechnical1} admits a \emph{unique} solution. This, together with \eqref{SetUci} and \eqref{SetUc}, indicates that the set $\mathcal{U}_c(0,u_d)$ is a singleton. We thus conclude from Theorem~\ref{ThInputSetCharacterization} that there uniquely exists $u_c \in \mathbb{I}_5^{0.2}(u_d)$ that establishes 
\begin{equation*}
	\restr{x_c (\cdot; x_0,u_c)}{[0,2]} \in \mathbb{I}_5^{0.2} \Big(\restr{x_d(\cdot;x_0,u_d)}{\llbracket 0,10\rrbracket}\Big).
\end{equation*}

The displacement $x$ and the velocity $v$ of the robot $\bm{\Sigma}_c$ are depicted in Figure~\ref{Outputs} (indicated by solid lines), whereas the continuous-time control $u_c \in \mathcal{U}_c(0,u_d)$ is depicted in Figure~\ref{Inputs} (indicated by solid lines). It is clear from Figure~\ref{Outputs} that the continuous-time control $u_c \in \mathcal{U}_c(0,u_d)$ enforces the robot $\bm{\Sigma}_c$ to evolve in accordance with the STL formula $\mathscr{S}_c$, \textit{i.e.,} $u_c \in \mathcal{U}_c(0,u_d)$ is such that \eqref{Technical8} holds. Particularly, it can be observed from Figure~\ref{Displacement} that the displacement $x$ reaches $-2 \unit{m}$ at $0.6$ and $2.0$ seconds, whereas it surpasses $2\unit{\meter}$ over the interval $[1.2,1.4]$. According to Figure~\ref{Velocity}, on the other hand, the velocity $v$ always remains between $-15\unit{\meter/\second}$ and $15\unit{\meter/\second}$. 
\begin{figure}
	\centering
	\includegraphics[width=0.5\textwidth]{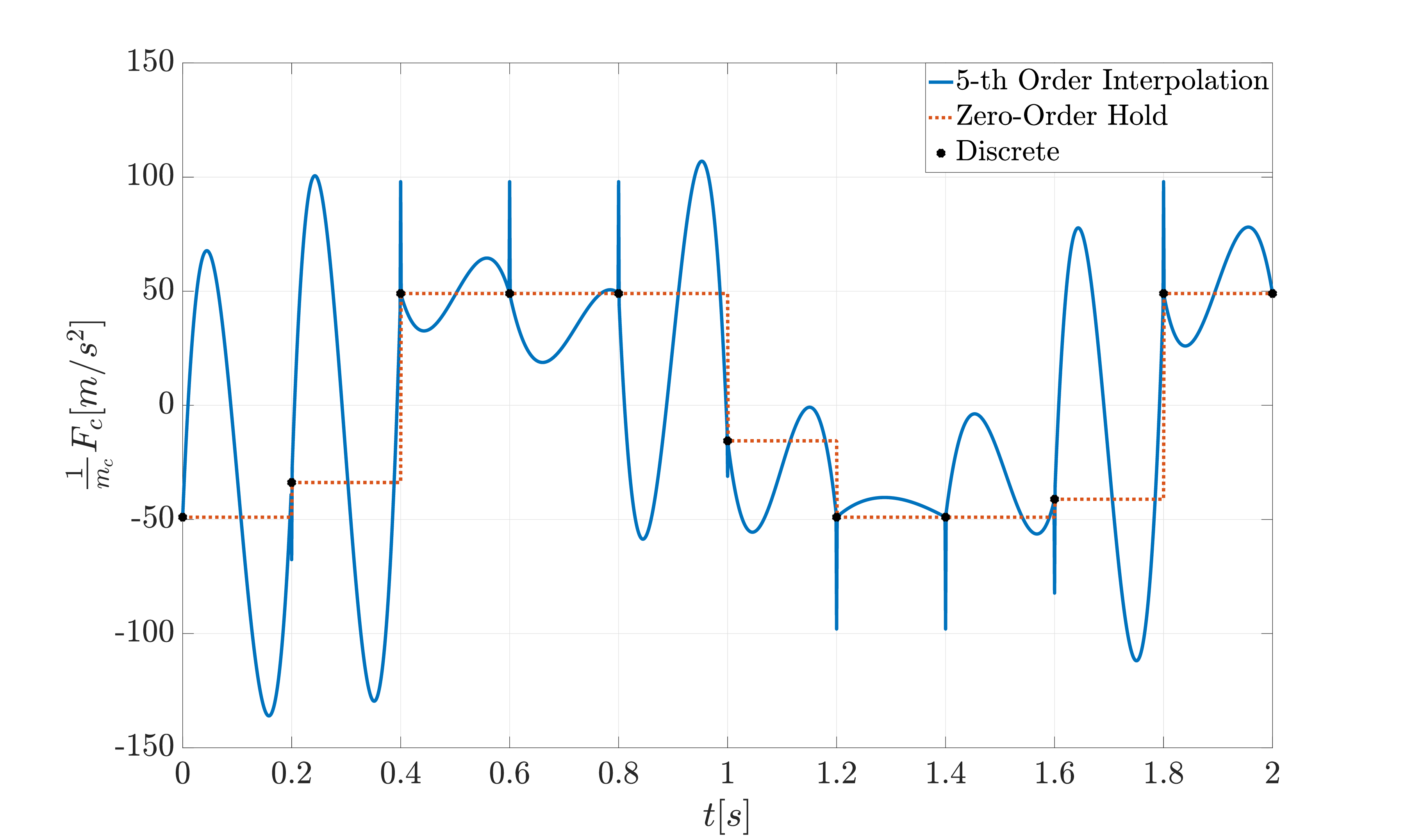}
	\caption{Subject to the continuous-time control $u_c\in\mathcal{U}_c(0,u_d)$ (indicated by solid lines), the robot $\bm{\Sigma}_c$ experiences a displacement $x$ (Figure~\ref{Displacement}) and a velocity $v$ (Figure~\ref{Velocity}) that adhere to the requirements specified by the formula $\mathscr{S}_c$. However, subject to a continuous-time control $u_c$ obtained by ZOH interpolation of the control sequence $u_d$ (indicated by dashed lines), the displacement $x$ of the robot does not comply with formula $\mathscr{S}_c$.}
	\label{Inputs}
\end{figure}

To show the efficacy of control synthesis using system interpolation, we have also depicted the displacement $x$ and the velocity $v$ of the robot $\bm{\Sigma}_c$ (indicated by dashed lines) in Figure~\ref{Outputs}, subject to a continuous-time control $u_c$ obtained as a zero-order hold (ZOH) interpolation of the control sequence $u_d$ (indicated by dashed lines in Figure~\ref{Inputs}). It is clear from Figure~\ref{Outputs} that subject to such control, the robot does not fulfill the requirements specified by the formula $\mathscr{S}_c$. Particularly, in this case, the velocity surpasses $-15\unit{\meter/\second}$ (between $1.6$ and $2.0$ seconds), while the displacement $x$ never reaches $2$ (between $1$ and $1.4$ seconds). 
\begin{remark}
	In this numerical example, the control $u_c \in \mathcal{U}_c(x_0,u_d)$ enforced the continuous-time system $\bm{\Sigma}_c$ to \emph{fully} satisfy the STL formula $\mathscr{S}_c$. In general, however, choosing $u_c \in \mathcal{U}_c(x_0,u_d)$ does not guarantee the \emph{full} satisfaction of an STL formula, but rather it only guarantees satisfaction at instances corresponding to the sampling time $\tau$. Nevertheless, such choice enables us to measure the extent of (possible) specification violation in between these instances by making use of \eqref{DistCondition}. In fact, by choosing $u_c \in \mathcal{U}_c(x_0,u_d)$, we guarantee that the divergence of state values from regions of interest (which are dictated by the formula) always remain within the bound specified by \eqref{DistCondition}.\demo
	\end{remark}
\section{Conclusion}\label{Sec_Conclusion}
In this paper, we developed a formal framework for comparison of systems across different time domains. We accomplished this by introducing the notion of system interpolation, which determines whether the input-state trajectories of a continuous-time system can be realized as piecewise polynomial interpolations of input-state trajectories of a discrete-time system. Here, we characterized piecewise polynomial interpolation of a discrete-time signal as a continuous-time function such that, at given sampling instants, it coincides with the discrete-time signal and that, over each interval between consecutive instants, it is represented by a polynomial of a prescribed degree. We then showed that in order to investigate system interpolation, it suffices to check whether the continuous-time dynamics can generate piecewise polynomial interpolations of the discrete-time input-state trajectories over an interval that corresponds to a single (discrete) time step. Accordingly, by representing these piecewise polynomial interpolations as linear combinations of shifted Legendre polynomials (up to a prescribed degree), we characterized system interpolation as a subspace inclusion that is solely in terms of the parameters of the continuous-time and discrete-time systems. Such characterization is computationally efficient as it can be formulated as a simple rank condition. Afterwards, for a given discrete-time input-state trajectory, we characterized the class of all inputs that enforce the continuous-time system to admit input-state trajectories that are realized as piecewise polynomial interpolations of this discrete-time trajectory. Subsequently, we exploited system interpolation to discretize a given continuous-time dynamics into a discrete-time one. Lastly, for a given specification, we employed system interpolation for control synthesis to ensure that the controlled continuous-time trajectories adhere to the specification at each sampling instants and to measure the extent of (possible) specification violation over intervals between these instances. 

For the future, we aim to employ the notion of system interpolation to conduct (signal/linear) temporal logic control for continuous-time systems. Particularly, our goal is to combine system interpolation with temporal logic design techniques developed for discrete-time systems to provide guarantees on continuous-time ones that interpolate them. 
\bibliographystyle{IEEEtran}
\bibliography{Reference.bib}
\begin{IEEEbiography}[{\includegraphics[width=1in,height=1.25in,clip,keepaspectratio]{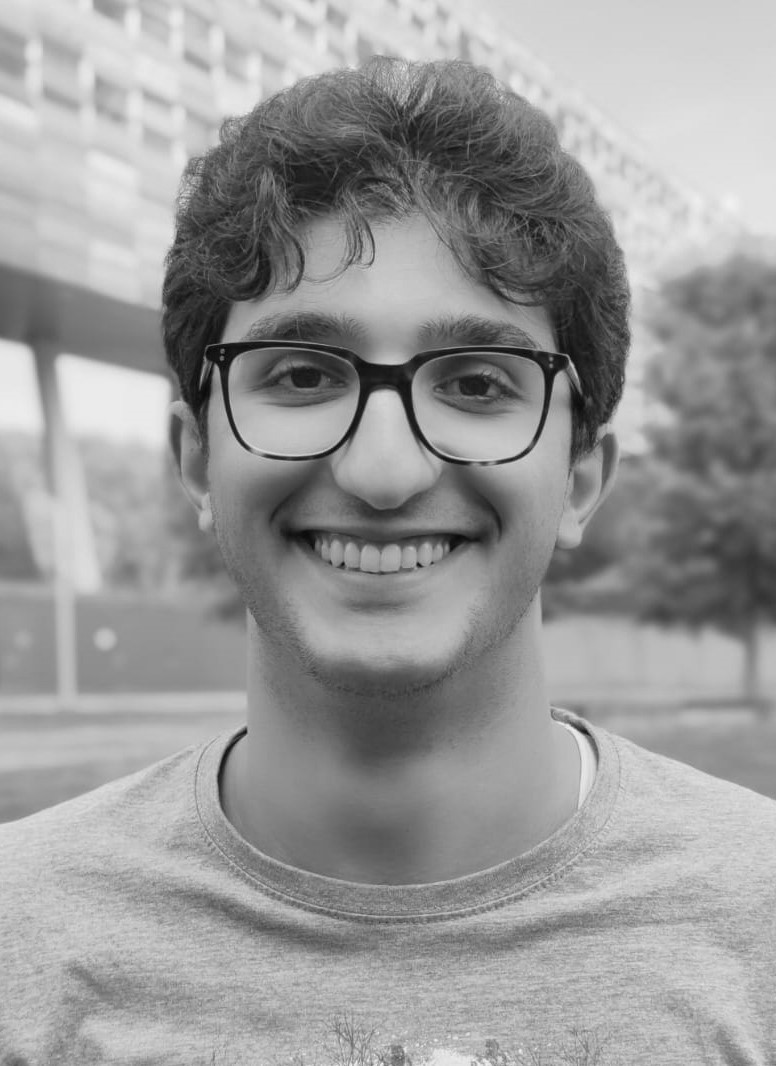}}]{Armin Pirastehzad} (Member, IEEE) earned his Ph.D. in Applied Mathematics, specializing in Systems and Control, from the University of Groningen, the Netherlands, in 2025. He previously obtained his B.Sc. and M.Sc. degrees in Electrical Engineering, also specializing in Systems and Control, from the University of Tehran, Iran, in 2017 and 2020, respectively. He is currently a Postdoctoral Researcher at the Bernoulli Institute for Mathematics, Computer Science, and Artificial Intelligence, University of Groningen.
\end{IEEEbiography}
\begin{IEEEbiography}[{\includegraphics[width=1in,height=1.25in,clip,keepaspectratio]{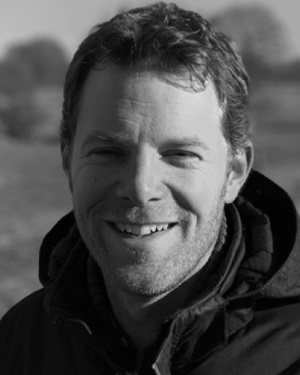}}]{Bart Besselink} (Senior Member, IEEE) received the M.Sc. (cum laude) degree in mechanical engineering in 2008 and the Ph.D. degree in 2012, both from Eindhoven University of Technology, Eindhoven, The Netherlands.
	
	Since 2016, he has been with the Bernoulli Institute for Mathematics, Computer Science and Artificial Intelligence, University of Groningen, Groningen, The Netherlands, where he is currently an associate professor. He was a short-term Visiting Researcher with the Tokyo Institute of Technology, Tokyo, Japan, in 2012. Between 2012 and 2016, he was a Postdoctoral Researcher with the ACCESS Linnaeus Centre and Department of Automatic Control, KTH Royal Institute of Technology, Stockholm, Sweden.
	
	His main research interests are on mathematical systems theory for large-scale interconnected systems, with emphasis on contract-based design and control, compositional analysis, model reduction, and applications in intelligent transportation systems and neuromorphic computing. He is a recipient (with Xiaodong Cheng and Jacquelien Scherpen) of the 2020 Automatica Paper Prize.
\end{IEEEbiography}
\end{document}